%% file: main.tex
\documentclass[11pt,a4paper]{amsart}

\include{mymacros}

\title{The geometry of groups containing almost normal subgroups}
\author{Alexander  Margolis}
\address{Alexander J. Margolis, Mathematics Department, Technion - Israel Institute of Technology, Haifa, 32000, Israel}
\email{amargolis@campus.technion.ac.il}
\thanks{This research was supported by the Israel Science Foundation (grant No. \texttt{1026/15}).}
\begin{document}

\begin{abstract}
A subgroup $H\leq G$ is said to be \emph{almost normal} if every conjugate of $H$ is commensurable to $H$. If $H$ is almost normal, there is a well-defined \emph{quotient space} $G/H$. We show that if a group $G$ has type $F_{n+1}$ and contains an almost normal coarse $PD_n$ subgroup $H$ with $e(G/H)=\infty$, then whenever $G'$ is quasi-isometric to $G$ it contains an almost normal subgroup $H'$ that is quasi-isometric to $H$. Moreover,  the quotient spaces $G/H$ and $G'/H'$ are quasi-isometric. This generalises a theorem of  Mosher--Sageev--Whyte, who prove the case in which  $G/H$ is quasi-isometric to a finite valence bushy tree.  
 Using work of Mosher, we generalise a result of Farb--Mosher to show that for many surface group extensions $\Gamma_L$,  any group quasi-isometric to  $\Gamma_L$ is virtually isomorphic to $\Gamma_L$.  We also prove quasi-isometric rigidity for the class of finitely presented $\mathbb{Z}$-by-($\infty$ ended) groups.
\end{abstract}
\maketitle

\input{intro}
\input{preliminaries}

\input{bundles}

\input{maintheorem}

\input{central}
\appendix
\input{fibreqirigidity}

\input{products}
\bibliography{bibliography/bibtex} 
\bibliographystyle{amsalpha}
\end{document}

%% file: mymacros.tex
\usepackage{amssymb,amscd,amsmath,amsthm}
\usepackage{microtype}
\usepackage{graphicx}
\usepackage{scalerel}
\usepackage{mathtools}
\usepackage{hyperref}
\usepackage[shortlabels]{enumitem}
\usepackage{tikz-cd}
\usepackage{csquotes}
\usepackage{thmtools}

\newtheorem{thm}{Theorem}[section]
\newtheorem{cor}[thm]{Corollary}
\newtheorem{lem}[thm]{Lemma}
\newtheorem{prop}[thm]{Proposition}
\newtheorem*{thm*}{Theorem}
\newtheorem*{prop*}{Proposition}
\theoremstyle{definition}
\newtheorem{defn}[thm]{Definition}
\newtheorem*{defn*}{Definition}
\theoremstyle{remark}
\newtheorem{rem}[thm]{Remark}
\newtheorem{exmp}[thm]{Example}
\newtheorem{ques}[thm]{Question}
\newtheorem*{ques*}{Question}

\newfont{\cusfont}{alnorm10}
\newcommand{\alnorm}{\text{\;\cusfont Q\; }}
\newcommand{\QSym}{\mathrm{QSym}(S^1)}
\newcommand{\id}{\mathrm{id}}
\newcommand{\Fix}[1]{\mathrm{Fix}(#1)}
\DeclareMathOperator{\Aut}{Aut}
\newcommand{\cC}{\mathcal{C}}
\DeclareMathOperator{\im}{im}
\DeclareMathOperator{\QI}{QI}
\DeclareMathOperator{\Out}{Out}
\DeclareMathOperator{\PSL}{PSL}
\DeclareMathOperator{\MCG}{MCG}
\DeclareMathOperator{\Norm}{Norm}

\DeclareMathOperator{\Isom}{Isom}
\DeclareMathOperator{\Homeo}{Homeo}
\DeclareMathOperator{\Comm}{Comm}
\DeclareMathOperator{\QIsom}{QIsom}
\newcommand{\ol}{\overline}
\newcommand{\Solv}{\textsc{Solv}}
\newcommand{\Haus}{\text{Haus}}
\newcommand{\bs}{\backslash}
\newcommand{\bZ}{\mathbb{Z}}
\newcommand{\bR}{\mathbb{R}}
\newcommand{\cT}{\mathcal{T}}
\newcommand{\cF}{\mathcal{F}}
\newcommand{\cO}{\mathcal{O}}
\newcommand{\bN}{\mathbb{N}}
\newcommand{\cH}{\mathcal{H}}
\newcommand{\cW}{\mathcal{W}}
\newcommand{\cS}{\mathcal{S}}
\newcommand{\Hyp}{\mathbb{H}^2}
\newcommand{\MF}{\mathcal{MF}}
\newcommand{\PMF}{\mathbf{P}\mathcal{MF}}
\newcommand{\bP}{\mathbf{P}}
\newcommand{\bX}{\mathbf{X}}

\newcommand\geodesic[2]{\overleftrightarrow{(#1,#2)}}

%% file: intro.tex
\section{Introduction}

A central idea in geometric group theory is that a finitely generated group equipped with the word metric is a geometric object in its own right. This metric is well-defined up to quasi-isometry.  A typical question is the following: given a class $\cC$ of finitely generated groups, is it true that any finitely generated group quasi-isometric to an element of  $\cC$ is also in $\cC$?  Positive answers to this question occur surprisingly often, and this phenomenon is  called  \emph{quasi-isometric rigidity}.
A related problem is \emph{quasi-isometric classification}:  when are groups in $\cC$ quasi-isometric? These questions comprise Gromov's program of studying finitely generated groups up to quasi-isometry.

One approach to these questions is to show that if a group $G$ can be decomposed into ``smaller'' or ``simpler'' groups, then any group quasi-isometric to $G$ also decomposes in a similar way.  The aim of this article is to investigate instances in which the following question and generalisations of it are true:
\begin{ques}\label{ques:normal subgp qiinvariant}
Suppose a finitely generated group $G$ contains an infinite normal subgroup $H$.
If  $G'$ is a finitely generated group  quasi-isometric to $G$, does it contain a normal subgroup $H'$ such that $H$ is  quasi-isometric to $H'$ and $G/H$ is quasi-isometric to $G'/H'$?
\end{ques}
\noindent Although this question is  false in full generality,  we can nonetheless give several instances for which there is a  positive answer. These are applications of Theorem \ref{thm:main intro}, the main result of this article.
We say that a  group is \emph{$\mathbb{Z}$--by--($\infty$ ended)} if it has an infinite cyclic normal subgroup such that the quotient is infinite ended. Our first result says that the class of finitely presented $\mathbb{Z}$--by--($\infty$ ended) groups is quasi-isometrically rigid:
\begin{restatable*}{thm}{zfibre}\label{thm:zfibre}
Let $G$ be a finitely presented $\mathbb{Z}$--by--($\infty$ ended) group.  If $G'$ is any finitely generated group quasi-isometric to $G$, it is also $\mathbb{Z}$--by--($\infty$ ended). Moreover, any quasi-isometry $f:G\rightarrow G'$ induces  a quasi-isometry $G/\mathbb{Z}\rightarrow G'/\mathbb{Z}$ between  quotient groups.
\end{restatable*}

A group $G$ is  of type $F_{n}$ if it has a $K(G,1)$ with finite $n$-skeleton.
We say that  $G$ is \emph{$\mathbb{Z}^n$--by--($\infty$ ended) with almost injective quotient (AIQ)} if  $G$ has a normal subgroup $H\cong \mathbb{Z}^n$ such that the quotient is infinite ended and  the natural homomorphism $G/H\rightarrow \mathrm{Aut}(H)\cong GL(n,\mathbb{Z})$ has finite kernel. We show that this class of groups is quasi-isometrically rigid:
\begin{restatable*}{thm}{znfibre}\label{thm:zn fibre with aiq}
Let $G$ be a group of type $F_{n+1}$ that is $\mathbb{Z}^n$--by--($\infty$ ended) with AIQ.  If $G'$ is any finitely generated group quasi-isometric to $G$, then it is also $\mathbb{Z}^n$--by--($\infty$ ended) with AIQ. Moreover, any quasi-isometry $f:G\rightarrow G'$ induces  a quasi-isometry $G/\mathbb{Z}^n\rightarrow G'/\mathbb{Z}^n$ between the infinite ended quotients.
\end{restatable*}

However, Theorem \ref{thm:zn fibre with aiq} does not hold if we relax the AIQ hypothesis. For instance, in Example \ref{exmp:qitofreeXz2} we describe a group quasi-isometric to $\mathbb{Z}^2\times F_2$  that does not (virtually) contain a free abelian normal subgroup of rank 2. Similar examples are also considered in work of  Leary--Minasyan \cite{learyminasyan19}. This resolves a  question of \cite[Section 12.2]{frigerio2015graphmflds}.

Using work of Mosher \cite{mosher2003fiber}, we can also prove a much stronger form of  quasi-isometric rigidity for certain surface group extensions. The Dehn--Nielsen--Baer theorem says that   $\Out(\pi_1(S))\cong \MCG(S)$, where  $\MCG(S)$ is the extended mapping class group of a  closed hyperbolic surface $S$. Given a  subgroup $L\leq \mathrm{MCG}(S)$, let $\Gamma_L$ be the associated surface group extension $1\rightarrow \pi_1(S)\rightarrow \Gamma_L\rightarrow L\rightarrow 1$. We say that $L$ is \emph{irreducible} if it doesn't preserve a finite collection of disjoint simple closed curves. 

\begin{restatable*}{thm}{qirsurfex}\label{thm:qirigidity surface group extension}
Suppose $L$ is an irreducible subgroup  of $\mathrm{MCG}(S)$ that is of type $F_3$ and has infinitely many ends. Let $\Gamma_L$ be the associated surface group extension. If $G$ is any finitely generated group quasi-isometric to $\Gamma_L$, then there is a finite normal subgroup $N\vartriangleleft G$ such that $G/N$ is abstractly commensurable to $\Gamma_L$, i.e. $\Gamma_L$ and $G/N$ have isomorphic finite index subgroups. 
\end{restatable*}

This generalises a result of Farb--Mosher \cite{farbmosher2002surfacebyfree}, who prove the  case in which $L$ is a convex-cocompact infinite-ended free group. A theorem of Koberda gives an abundance of suitable subgroups  $L\leq \mathrm{MCG}(S)$ to which Theorem \ref{thm:qirigidity surface group extension} can be applied   \cite{koberda2012RAAGSinMCGS}. Koberda shows that such subgroups are generic: if $n>1$ and $g_1,\dots, g_n$ are infinite order elements of $\MCG(S)$ that do not share a common power, at least one of which is a pseudo-Anosov, then $\langle g_1^N,\dots g_n^N\rangle$ is an infinite-ended, irreducible subgroup of $\MCG(S)$ that is of  type $F_3$ for $N$ sufficiently large.

However, there are many situations in which Question \ref{ques:normal subgp qiinvariant} fails to be true. In addition to Example \ref{exmp:qitofreeXz2} described above, we recall that Burger and Mozes constructed simple groups that are quasi-isometric to $F_2\times F_2$ \cite{burgermozes97simple}. Theorem \ref{thm:main intro} demonstrates that a weakening of the notion of normal subgroup is frequently a quasi-isometry invariant.  We now describe this property.

Given an ambient group $G$, two subgroups $H$ and $K$ are said to be \emph{commensurable} if the intersection $H\cap K$ has finite index in both $H$ and $K$. We define  $\mathrm{Comm}_G(H)$ to be the subgroup of  all $g\in G$ such that  $H$ and $gHg^{-1}$ are commensurable. A subgroup $H\leq G$ is said to be \emph{almost normal} if $G=\mathrm{Comm}_G(H)$, i.e. $H$ and $gHg^{-1}$ are commensurable for all $g\in G$. We denote this by $H \alnorm  G$. Almost normal subgroups have been studied extensively and are also known as inert subgroups, near normal subgroups and commensurated subgroups, e.g. \cite{belyaev93inert,kropholler2006spectral,shalomwillis13}.

Given a finitely generated group $G$ and a subgroup $H\alnorm G$, the set $G/H$ of left  $H$-cosets carries a proper metric, unique up to quasi-isometry, such that the natural left $G$ action is isometric. We refer to Section \ref{sec:coarse bundles} for a definition of this metric.   This space $G/H$ is called the \emph{quotient space}.  Indeed, in the case where $H$ is normal, the quotient space is quasi-isometric to the quotient group equipped with the word metric. 
 
The map $p:G\rightarrow G/H$ taking $g$ to the left coset $gH$ is a \emph{coarse bundle} in the sense of \cite{whyte2010coarse}.   
 Showing that quasi-isometries preserve this coarse bundle structure is a key step in several celebrated quasi-isometric rigidity results, e.g. \cite{farbmosher1999bs2},  \cite{farbmosher2000abelianbycyclic}, \cite{whyte2001baumslag}, \cite{farbmosher2002surfacebyfree}, \cite{mosher2003quasi} and \cite{eskinfisherwhyte12coarse}. This property is also called  horizontal--respecting or height--respecting in the literature.
 
We consider the following variant of Question \ref{ques:normal subgp qiinvariant}:
\begin{ques}\label{ques:alnormal subgp qiinvariant}
Let  $G$ be a finitely generated group containing an infinite almost normal subgroup $H$. If  $G'$ is a finitely generated group  quasi-isometric to $G$, does $G'$ also contain an almost normal subgroup $H'\alnorm G'$ such that $H$ is  quasi-isometric to $H'$  and $G/H$ is quasi-isometric to $G'/H'$?
\end{ques}
The Burger--Mozes groups, which were counterexamples to Question \ref{ques:normal subgp qiinvariant}, are not counterexamples to Question \ref{ques:alnormal subgp qiinvariant}. Indeed, the groups considered in \cite{burgermozes97simple} are of the form $\Gamma= F*_HK$, where $H$, $F$ and $K$ are finitely generated free groups such that $H$ has finite index in both $F$ and $K$. It is easy to see that  $H$ is almost normal in $\Gamma$. The following proposition, which can be deduced from Theorem \ref{thm:qi to product contains alnorm},   says that this is true, up to finite index, for all groups quasi-isometric to a product of non-abelian free groups. 
\begin{prop}\label{prop:gps qi to TxT}
Suppose $G$ is a finitely generated group quasi-isometric to $F_2\times F_2$. Then $G$ contains a finite index subgroup $G'$ and a  finitely generated non-abelian free subgroup $H$ such that  $H\alnorm G'$ and  the quotient space $G'/H$ is quasi-isometric to $F_2$. 
\end{prop}
\noindent  This motivates the study of almost normal subgroups and their quotient spaces  as a means to understanding the coarse geometry of finitely generated groups.

To show quasi-isometric rigidity of almost normal subgroups, we apply the coarse topological techniques developed in  \cite{margolis2018quasi}. To do this, we need to assume that the almost normal subgroup is a \emph{coarse Poincar\'e duality group}. 
These were defined in \cite{kapovich2005coarse} and used extensively in the work of Mosher--Sageev--Whyte \cite{mosher2003quasi,mosher2011quasiactions} and Papasoglu \cite{papasoglu2007group}. We do not give a definition here, but note that the archetypal example of a coarse Poincar\'e duality group of dimension $n$, or coarse $PD_n$ group, is the fundamental group of a closed aspherical $n$-manifold. In particular, they include finitely generated free abelian groups and more generally, virtually polycyclic groups. The class of coarse Poincar\'e duality groups is closed under quasi-isometries.

We demonstrate that under suitable hypotheses, if a group contains an almost normal coarse $PD_n$ subgroup, then any group quasi-isometric to it also contains such a subgroup, and  the associated coarse bundle structure is preserved by quasi-isometries. 
\begin{thm}\label{thm:main intro}
Fix $n\geq 1$. Let $G$ be a group of type $F_{n+1}$ and  $H\alnorm G$ be a coarse $PD_n$ subgroup with $e(G/H)\geq 3$. Suppose  $G'$ is a finitely generated group quasi-isometric to $G$. Then  $G'$ contains an almost normal coarse $PD_n$ subgroup   $H'$ such that $H$ is quasi-isometric to $H'$ and  $G/H$ is quasi-isometric to  $G'/H'$.
\end{thm}

Theorem \ref{thm:main intro} is the main result of this article. We hope that this theorem and the methods used to prove it  will have many more applications in quasi-isometric rigidity and classification results.
Theorem \ref{thm:main intro} will be deduced from  Theorem  \ref{thm:main fib bundle}, a more technical and quantitative statement that is needed to prove the preceding applications.  Theorem \ref{thm:main intro} and many of its applications actually hold under the weaker assumption that $G$ is only of type $FP_{n+1}$, a homological analogue of type $F_{n+1}$.

Recall that a subgroup $H\leq G$ is said to be characteristic if it is preserved by every automorphism of $G$. Analogously, we say that a subgroup $H\leq G$ of a finitely generated group is \emph{coarsely characteristic} if it is coarsely preserved by every quasi-isometry of $G$, i.e. for every quasi-isometry $f:G\rightarrow G$, $H$ and $f(H)$ are at finite Hausdorff distance. Every coarsely characteristic subgroup is easily seen to be almost normal, see Proposition \ref{prop:coset characterisation}. A key step in our proof of Theorem \ref{thm:main intro} is a partial converse:
\begin{thm}[c.f. Lemma \ref{lem:coarsely characteristic}]
If $G$ is a group of type $F_{n+1}$ and  $H\alnorm G$ is a coarse $PD_n$ subgroup with $e(G/H)\geq 3$, then $H$ is coarsely characteristic.
\end{thm}
The $e(G/H)\geq 3$ hypothesis in Theorem \ref{thm:main intro} is needed to apply the coarse topological techniques of \cite{margolis2018quasi} and deduce that $H$ is coarsely characteristic. Indeed, both  $\mathbb{Z}^2$ and $\mathbb{Z}^3$ have an infinite cyclic normal subgroup that is not coarsely characteristic. Similarly, if $M$ is a fibred hyperbolic 3-manifold, then the normal subgroup $\pi_1(S)\vartriangleleft \pi_1(M)$ is not coarsely characteristic.

Theorem \ref{thm:main intro}  generalises  \cite[Theorem 2]{mosher2003quasi}, which proves the special case where the quotient space $G/H$ is quasi-isometric to a finite valence tree with infinitely many ends. To see this, we explain how to reformulate the hypotheses of Theorem \ref{thm:main intro} in terms of graphs of groups. 

We first make the following observation: if a finitely generated group $G$ contains a normal subgroup $H\vartriangleleft G$ such that the quotient $K=G/H$ has more than one end, then $G$ splits as an amalgamated free product or HNN extension over a finite extension of $H$. This can be seen by applying Stallings' theorem to the quotient. If $K$ is assumed to be finitely presented, then by applying 
\cite{dunwoody1985accessibility} to $K$, we see that $G$ splits as a finite graph of groups whose edge groups are finite extensions of $H$ and   no vertex group splits over a finite extension of $H$. A similar statement holds in the situation when $H$ is only assumed to be almost normal:

\begin{restatable*}{thm}{reldunaccessibility}\label{thm:gog}
If $G$ is finitely presented  and  $H\alnorm G$ is finitely generated, then  $G$ is the fundamental group of a finite graph of groups such that: 
\begin{enumerate}
\item every edge group is commensurable to $H$;
\item every vertex group is finitely generated and doesn't split over a subgroup commensurable to $H$. 
\end{enumerate}
\end{restatable*}

This can be deduced by applying Dunwoody accessibility  to the quotient space $G/H$ \cite{dunwoody1985accessibility}. Theorem \ref{thm:gog} demonstrates a form of accessibility  over a family of subgroups that are not necessarily small in the sense of \cite{bestvinafeighn91accessibility}. Similar  applications of Dunwoody's accessibility theorem to non-proper actions appear in the setting of totally disconnected locally compact groups \cite{kronmoller08roughcayley}. 
 We can use  Theorem \ref{thm:gog} to reformulate Theorem \ref{thm:main intro}: 
\begin{restatable*}{cor}{gogs}\label{cor:qirigidity of gogs}
Let $G$ be a group of type $F_{n+1}$. Suppose that $G$ is the fundamental group of a graph of groups $\mathcal{G}$ with the following properties:
\begin{enumerate}
\item \label{item:condition bushy} the associated Bass-Serre tree has at least three ends;
\item \label{item:conjugates are commensurable} all conjugates of all edge groups are coarse $PD_n$ groups and are commensurable to one another;
\item \label{item:condition vertex group doesn't split} each vertex group is finitely generated and doesn't split over a subgroup commensurable to one of its incident edge groups.
\end{enumerate}
If $G'$ is a finitely generated  group  quasi-isometric to $G$, it is also the fundamental group of a graph of groups satisfying (\ref{item:condition bushy})--(\ref{item:condition vertex group doesn't split}).
\end{restatable*}

We can also prove a partial quasi-isometric classification theorem. If $\mathcal{G}$ is a graph of groups  satisfying (\ref{item:condition bushy})--(\ref{item:condition vertex group doesn't split}) as above, a vertex group is said to be \emph{essential} if it is not commensurable to an incident edge group. If a vertex group $G_v$ is essential, any incident edge group $G_e$ is almost normal in $G_v$ with $e(G_v/G_e)\geq 1$. A quasi-isometry $f:G_v\rightarrow G'_{v'}$ is \emph{fibre-preserving} if whenever $G_e$ and $G'_{e'}$ are edge groups incident to $G_v$ and $G'_{v'}$, $f$ sends left cosets of $G_e$ to within uniform finite Hausdorff distance of left cosets of $G'_{e'}$. In particular, if  $f$ is fibre-preserving it induces a quasi-isometry $G_v/G_e \rightarrow G'_{v'}/G'_{e'}$ between quotient spaces.
\begin{restatable*}{thm}{qiclassification}\label{thm:qiclassification of gogs}
Let $G$ and $G'$  be groups of type $F_{n+1}$  that are fundamental groups of  finite graphs of groups $\mathcal{G}$ and $\mathcal{G'}$ satisfying conditions (\ref{item:condition bushy})--(\ref{item:condition vertex group doesn't split}) of Corollary \ref{cor:qirigidity of gogs}. 
Suppose that $G$ and $G'$ are quasi-isometric. Then for each essential vertex group  $G_v$ of $\mathcal{G}$, there is an essential vertex group  $G'_{v'}$ of $\mathcal{G}'$ and a fibre-preserving quasi-isometry $G_v\rightarrow G'_{v'}$. Conversely, for each essential vertex $G'_{v'}$ of $\mathcal{G}'$ there exists an essential vertex group $G_v$ of $\mathcal{G}$ and a fibre-preserving quasi-isometry $G_v\rightarrow G'_{v'}$.
\end{restatable*}
\noindent A complete quasi-isometric  classification for such groups  is likely to be very difficult and is beyond the scope of this article. However,  Farb--Mosher and  Whyte have obtained  quasi-isometric classification results in special cases \cite{farbmosher1998bs1,farbmosher2000abelianbycyclic, farbmosher2002surfacebyfree,whyte2001baumslag,whyte2010coarse}. 

A \emph{solvable Baumslag--Solitar group} is a group of the form $BS(1,n)=\langle a,t\mid tat^{-1}=a^n\rangle$ for some $n>1$.  An important application of Theorem \ref{thm:main intro} is quasi-isometric rigidity for solvable Baumslag--Solitar groups, originally proven by Farb and Mosher \cite{farbmosher1999bs2}.
\begin{thm}[\cite{farbmosher1999bs2}]
 A finitely generated group quasi-isometric to a solvable Baumslag--Solitar group  has a finite index subgroup isomorphic to a solvable Baumslag--Solitar group.
\end{thm}
\begin{proof}
Using Corollary \ref{cor:qirigidity of gogs} and Theorem \ref{thm:qiclassification of gogs} we deduce that as $BS(1,n)$ is the fundamental group of a graph of two-ended groups, so is $G$.  Moreover,  $BS(1,n)$ has exponential growth and is solvable. Thus $G$ has exponential growth and is amenable. In particular, $G$ is not virtually $\bZ^2$ and cannot contain a non-abelian free group.  Thus $G$ is an ascending HNN extension of a 2-ended group, so contains a finite index solvable Baumslag-Solitar subgroup.
\end{proof} 
We claim no originality for the preceding proof, which is well-known to experts and can be easily be deduced  from either \cite{mosher2003quasi}, \cite{papasoglu2005quasi} or \cite{papasoglu2007group} without the use of Corollary \ref{cor:qirigidity of gogs}. However, it is included to illustrate the power of the coarse topological techniques in this paper.

An important object in geometric group theory is the quasi-isometry group of a space, i.e. the group of all quasi-isometries modulo an appropriate equivalence relation. We can use the results of this paper to deduce the following:
\begin{cor}
Let $G$ and $H$ be as in Theorem \ref{thm:main intro}. Then there is a homomorphism $\mathrm{QI}(G)\rightarrow \mathrm{QI}(G/H)$.
\end{cor}
This sort of observation is the starting point in the calculation of quasi-isometry groups of solvable Baumslag--Solitar groups and surface-by-free groups in \cite{farbmosher1999bs2} and \cite{farbmosher2002surfacebyfree}.

We give an outline of the paper. Section \ref{sec:preliminaries} consists of background and preliminaries. In Section \ref{sec:coarse bundles} we introduce the notion of coarse bundles and investigate their coarse geometric properties.   In Section \ref{sec:mainthm} we prove Theorems \ref{thm:main intro} and \ref{thm:qiclassification of gogs}. Section \ref{sec:applications} introduces the \emph{fibre distortion function}, a quasi-isometry invariant that is used to prove Theorems \ref{thm:zfibre} and \ref{thm:zn fibre with aiq}. In Appendix \ref{sec:qirigidity surfaces} we give an account of Mosher's work on ``fibre-respecting quasi-isometries'' of surface group extensions \cite{mosher2003fiber}. We then combine Mosher's results with the main result of this article to deduce  Theorem \ref{thm:qirigidity surface group extension}.

Our results build on work of Vavrichek, who uses \cite{papasoglu2005quasi} to prove part of Theorem \ref{thm:main intro} in the case where $H$ is two-ended. Many of the ideas in this article are inspired by work of  Whyte \cite{whyte2001baumslag, whyte2010coarse}, particularly in Section  \ref{sec:applications}. 

The author would like to thank Panos Papasoglu,  Michah Sageev and Ian Leary for helpful conversations.

%% file: preliminaries.tex
\section{Preliminaries}\label{sec:preliminaries}
Much of this material is discussed in greater detail in \cite{margolis2018quasi}.
\subsection*{Coarse geometry}
Let $(X,d)$ be a metric space. For $\emptyset \neq A\subseteq X$ and $x\in X$, we set $d(x,A)\coloneqq \inf_{a\in A}d(x,a)$.   We define $N_r(A)\coloneqq\{x\in X\mid d(x,A)\leq r\}$. If $A=\{a\}$, we also denote $N_r(A)$ by  $N_r(a)$. We say that $X$ has \emph{bounded geometry} if there is a function $M:\mathbb{R}_{\geq 0}\rightarrow \mathbb{R}_{\geq 0}$ such that for all $r\in \mathbb{R}_{\geq 0}$ and $x\in X$, $\lvert N_r(x)\rvert\leq M(r)$.  
 The \emph{Hausdorff distance} between two subsets $A,B\subseteq X$ is defined to be \[d_\mathrm{Haus}(A,B)\coloneqq \inf\{r\in \mathbb{R}\mid A\subseteq N_r(B) \textrm{ and } B\subseteq N_r(A)\}.\] In general, this infimum is not achieved. However, the infimum  is achieved when $X$ is a finitely generated group equipped with the word metric, since then $d(X\times X)\subseteq \mathbb{N}$ is discrete.

Recall that a function $\phi:\mathbb{R}_{\geq 0} \rightarrow \mathbb{R}_{\geq 0}$ is \emph{proper} if the inverse images of compact sets are compact.

\begin{defn}
Let $(X,d_X)$ and $(Y,d_Y)$ be metric spaces and let $\eta,\phi:\mathbb{R}_{\geq 0} \rightarrow \mathbb{R}_{\geq 0}$ be proper non-decreasing functions. A function $f:X\rightarrow Y$ is a \emph{$(\eta,\phi)$-coarse embedding} if for all $x,x'\in X$, 
\[\eta(d_X(x,x'))\leq d_Y(f(x),f(x'))\leq \phi(d_X(x,x')).\] We say that $f$ is a \emph{coarse embedding}  if it is an $(\eta,\phi)$-coarse embedding for some  $\eta$ and $\phi$. We say that $\eta$ and $\phi$ are the \emph{distortion functions} of $f$.
A \emph{coarse equivalence} is a  coarse embedding $f:X\rightarrow Y$ such that $N_A(f(X))=Y$ for some $A\geq 0$.
\end{defn}

\begin{rem}\label{rem:properinv}
For each proper non-decreasing function $\phi:\mathbb{R}_{\geq 0}\rightarrow \mathbb{R}_{\geq 0}$, we define another proper non-decreasing function $\widetilde{\phi}:\mathbb{R}_{\geq 0}\rightarrow \mathbb{R}_{\geq 0}$ by $\widetilde{\phi}(R):=\sup(\phi^{-1}([0,R]))$.  This can be thought of as a sort of inverse to $\phi$ in the sense that if $\phi(S)\leq R$, then $S\leq \widetilde \phi(R)$, and if $R<\phi(S)$, then $\widetilde \phi(R)\leq S$.
\end{rem}

We are particularly interested in the following subclass of coarse embeddings:
\begin{defn}
Let $(X,d_X)$ and $(Y,d_Y)$ be metric spaces and let $K\geq 1$ and $A\geq 0$. A function $f:X\rightarrow Y$ is a \emph{$(K,A)$-quasi-isometry} if the following hold:
\begin{enumerate}
\item  for all $x,x'\in X$, 
$\frac{1}{K}d_X(x,x')-A \leq d_Y(f(x),f(x'))\leq K d_X(x,x')+A$;
\item for all $y\in Y$, there exists an $x\in X$ with $d_Y(f(x),y)\leq A$.
\end{enumerate} We say that $f$ is a \emph{quasi-isometry}  if there exist $K\geq 1$ and $A\geq 0$ such that $f$ is a $(K,A)$-quasi-isometry.
\end{defn}

We say that $f,g:X\rightarrow Y$ are \emph{$A$-close} if $\sup_{x\in X}d(f(x),g(x))\leq A$, and we say that $f$ and $g$ are \emph{close} if they are $A$-close for some $A<\infty$.  A \emph{coarse inverse} to $f:X\rightarrow Y$ is a function $\overline f:Y\rightarrow X$ such that $f\circ \overline f$ and $\overline f\circ f$ are close to $\id_Y$ and $\id_X$. Every quasi-isometry has a coarse inverse. Being close defines an equivalence relation on the set of all quasi-isometries from $X$ to $X$,  and we let $[f]$ denote the equivalence class containing $f$. We can thus define a group $\QI(X)\coloneqq \{[f]\mid f:X\rightarrow X \text{ is a quasi-isometry}\}$ with  group operation $[f].[g]=[f\circ g]$. 

A key idea in geometric group theory is that  a finitely generated group can be equipped with the word metric with respect to a finite generating set. It is an easy exercise to see that if we equip a group with two different word metrics with respect to two finite generating sets, these metrics are quasi-isometric. This motivates the study of groups up to quasi-isometry. Throughout this article, whenever a finitely generated group is considered as a metric space, it will always be equipped with the word metric with respect to some finite generating set unless explicitly stated.

The following lemma motivates our interest in more general coarse embeddings whose distortion functions are not necessarily affine.
\begin{lem}[e.g. see {\cite[Remark 1.20]{roe2003lectures}}]\label{lem:sbgpis coarsely embedding}
Let $G$ be a finitely generated group containing a finitely generated subgroup $H\leq G$. Let $d_H$  and $d_G$ be word metrics of $H$ and $G$ with respect to finite generating sets. Then the inclusion $(H,d_H)\rightarrow (G,d_G)$ is a coarse embedding.
\end{lem}

A space is said to be \emph{quasi-geodesic} if it is quasi-isometric to a geodesic metric space. For example, a finitely generated group $G$ equipped with the word metric is quasi-geodesic but not geodesic (unless $G$ is the trivial group). If a metric space is quasi-geodesic, it can be approximated by a simplicial complex known as the Rips complex.
\begin{defn}
Given a metric space $(X,d)$ and a parameter $r\geq 0$, the \emph{Rips complex} $P_r(X)$ is a simplicial complex with vertex set $X$ such that $\{x_0,\dots, x_n\}$ spans a simplex if  $d(x_i,x_j)\leq r$ for all $1\leq i,j\leq n$. The \emph{Rips graph} $P^1_r(X)$ is the 1-skeleton of $P_r(X)$.
\end{defn}

If $P^1_r(X)$  is connected, it can be equipped with the induced path metric in which all edges have length 1. The following proposition relates this metric to the original metric on $X$. A \emph{$t$-chain} of length $n$ from $x$ to $y$ is a sequence of points $x=x_0, x_1, \dots, x_n=y$ such that $d(x_{i-1},x_{i})\leq t$ for $1\leq i\leq n$.

\begin{prop}[special case of Proposition 2.5 of \cite{margolis2018quasi}]\label{prop:quasi geodesic}
Let $(X,d)$ be a metric space. The following are equivalent:
\begin{enumerate}
\item $X$ is a quasi-geodesic metric space;
\item there exist constants  $K\geq 1$ and $A,t\geq 0$ such that any $x,y\in X$ can be joined by a $t$-chain of length at most $Kd(x,y)+A$;
\item there exists a $t\geq 0$ such that for all $r\geq t$, the Rips graph $P_r^1(X)$ is connected and the inclusion $X\rightarrow P_r^1(X)$ is a quasi-isometry.
\end{enumerate}
\end{prop}
The following lemma gives an intrinsic characterisation of the image of a coarsely embedded quasi-geodesic metric space.
\begin{lem}\label{lem:coarse embedding}
Let $Y$ be a metric space and $W\subseteq Y$. The following are equivalent:
\begin{enumerate}
\item \label{lem:coarse embedding1} There is a quasi-geodesic metric space $X$, unique up to quasi-isometry, and a coarse embedding $f:X\rightarrow Y$ such that $d_\Haus(f(X),W)<\infty$.
\item \label{lem:coarse embedding2} There exists a $t\geq 0$ and a proper non-decreasing function $\eta:\bR_{\geq 0} \rightarrow \bR_{\geq 0}$ such that every $x,y\in W$ can be joined by a $t$-chain in $W$ of length at most $\eta(d(x,y))$.
\end{enumerate}
Moreover,  $t$ and $\eta$ in (\ref{lem:coarse embedding2}) depend only the distortion functions of $f:X\rightarrow Y$ and vice versa.  
\end{lem}
\begin{proof}
We equip $W\subseteq Y$  with the subspace metric.

(\ref{lem:coarse embedding1}) $\implies$ (\ref{lem:coarse embedding2}): Since $d_\Haus(f(X),W)<\infty$, we can modify $f$ by a bounded amount so that $f(X)\subseteq W$. Thus $f$ can be thought of as a coarse equivalence $f:X\rightarrow W$. Since $W$ is coarsely equivalent to a geodesic metric space,  Proposition 2.5 of \cite{margolis2018quasi} tells us that $W$ satisfies (\ref{lem:coarse embedding2}).

(\ref{lem:coarse embedding2}) $\implies$ (\ref{lem:coarse embedding1}): It follows from Proposition 2.5 of \cite{margolis2018quasi} and (\ref{lem:coarse embedding2}) that $W$ is coarsely equivalent to a (quasi-)geodesic metric space $X$, and so there is a coarse embedding $f:X\rightarrow Y$ with image $W$. If $W$ is also coarsely equivalent to a quasi-geodesic metric space $X'$, then $X$ and $X'$ are quasi-geodesic metric spaces that are coarsely equivalent. Since every coarse equivalence between quasi-geodesic metric spaces is a quasi-isometry, we conclude that $X$ and $X'$ are quasi-isometric.
\end{proof}
By a slight abuse of notation, we say that $W\subseteq Y$ is \emph{coarsely embedded} in $Y$ if it satisfies one of the two equivalent conditions in Lemma \ref{lem:coarse embedding}.

Most of the spaces that we work with are discrete geodesic metric spaces in the following sense:
\begin{defn}
Let $(X,d)$ be a metric space. A \emph{discrete geodesic} between $x,y\in X$ is a sequence of points $x=x_0,x_1,\dots, x_n=y$ such that $d(x_i,x_j)=\lvert i-j\rvert$.  We say that $X$ is a \emph{discrete geodesic metric space} if every pair of points can be joined by a discrete geodesic.
\end{defn}
\begin{rem}\label{rem:remetrise}
If $X$ is quasi-geodesic, it follows from Proposition \ref{prop:quasi geodesic} that $X$ can be remetrised by identifying it with the  vertex set of the Rips graph $P_r^1(X)$ for large enough $r$. This new metric is quasi-isometric to the original metric and is a discrete geodesic metric space. Thus every quasi-geodesic metric space can be remetrised to be a discrete geodesic metric space. Moreover, any group action by isometries on $X$ induces a group action by isometries on the remetrised discrete geodesic space. 
\end{rem}

We will require the following technical result:
\begin{prop}[cf. {\cite[Lemma 2.1]{farbmosher2000abelianbycyclic}}]\label{prop:fibreqi induces qi between fibres}
Let $X,Y,Z$ and $W$ be discrete geodesic metric spaces and let $s:Z\rightarrow G$ and $s':W\rightarrow Y$ be coarse embeddings.  Suppose that $f:X\rightarrow Y$ is a quasi-isometry and there is a constant $A$ such that $d_\mathrm{Haus}(f(\mathrm{im}(s)),\mathrm{im}(s'))\leq A$. Then there is a quasi-isometry $g:Z\rightarrow W$ with $\sup_{z\in Z}d_{Y}(s'\circ g(z),f\circ s(z))\leq A$, whose quasi-isometry constants depend only on $A$ and the distortion functions associated to $f$, $s$ and $s'$.
\end{prop}
\subsection*{Coarse separation and ends}
In what follows, we assume that $(X,d)$ is a discrete geodesic bounded geometry metric space. 
  Given a subset $C\subseteq X$, we define the coarse boundary $\partial C$ to be $\{x\in X\mid d(x,C)=1\}.$ We say that $C\subseteq X$ is a \emph{coarse complementary component} of $W\subseteq X$ if there exists an $A\geq 0$ such that $\partial C\subseteq N_A(W)$. The intersection, complement, union and symmetric difference of two coarse complementary components of $W$ is again a coarse complementary component of $W,$ as shown in \cite{margolis2018quasi}.

A coarse complementary component  is said to be \emph{deep} if it is not contained in $N_A(W)$ for any $A\geq 0$. Otherwise it is said to be \emph{shallow}. A collection  of  coarse complementary components $\{C_i\}_{i\in I}$ is said to be \emph{coarse disjoint} if  $C_i\cap C_j$ is shallow for every distinct $i,j\in I$. We say that $W$ \emph{coarsely $n$-separates} $G$ if there exist $n$ deep, coarse disjoint, coarse complementary components of $W$ in $G$. Coarse complementary components and coarse separation are quasi-isometry invariants. We caution the reader that coarse complementary components are not necessarily coarsely connected.

The following definition is slightly non-standard, but agrees with other definitions of ends.
\begin{defn}
Let $X$ be a discrete geodesic, bounded geometry metric space. We define the number of ends of $X$ to be \[e(X)\coloneqq \sup\{n\in \bN\mid \text{$X$ is $n$-separated by a point}\}.\]
\end{defn}

In the case where $W$ is a subgroup of $G$, coarse complementary components and coarse $n$-separation  can be characterised in terms of more classical notions of almost invariant subsets and Kropholler--Roller number of  relative ends, denoted $\tilde e(-,-)$ \cite{kropholler1989relative}. Although we do not define these notions here, the following proposition can be thought of as a definition  for the purposes of this article.

\begin{prop}[{\cite[Propositions 5.14 and 5.15]{margolis2018quasi}}]\label{prop:relends vs cccs}
Let $G$ be a finitely generated group and $H\leq G$ be a subgroup.
\begin{enumerate}
\item $C\subseteq G$ is a coarse complementary component of $H$ if and only if it is an $H$-almost invariant subset;
\item $H$ coarsely $n$-separates $G$ if and only if $\tilde e(G,H)\geq n$.
\end{enumerate}
\end{prop}
Bowditch's coends and Geoghegan's filtered ends are alternative  interpretations of relative ends of groups \cite{bowditch2002splittings,geoghegan2008topological}.
\subsection*{Finiteness properties and coarse topology}
We define finiteness properties of groups, which generalise the notion of being finitely generated and presented. All homology and cohomology will be taken with coefficients over $\mathbb{Z}_2$.
\begin{defn}
 We say that a group $G$ has \emph{type $F_n$} if there exists a $K(G,1)$ with finite $n$-skeleton.
We say that a group $G$ has \emph{type $FP_n$} over a ring $R$ if the trivial $RG$ module $R$ has a projective resolution $P_\bullet\rightarrow R$ with $P_i$ finitely generated for $i\leq n$.
\end{defn}Throughout this article, it will be assumed that type  $FP_n$ means type $FP_n$ over $\mathbb{Z}_2$. 
Finiteness properties can be characterised geometrically. To see this, we need the following definitions:

\begin{defn}
Let $X$ be a metric space and $n\geq 0$.
\begin{enumerate}
\item We say that $X$ is \emph{coarsely uniformly $n$-connected} if for all $i,r$, there exists a $j=j(i)\geq i$ and  $s=s(i,r)\geq r$ such that for all $k\leq n$ and $x\in X$, the map \[\pi_k(P_i(N_r(x)))\rightarrow \pi_k(P_j(N_s(x))),\] induced by inclusion, is zero.
\item We say that $X$ is \emph{coarsely uniformly $n$-acyclic} if for all $i,r$, there exists a $j=j(i)\geq i$ and  $s=s(i,r)\geq r$ such that for all $k\leq n$ and $x\in X$, the map \[\tilde H_k(P_i(N_r(x)))\rightarrow \tilde H_k(P_j(N_s(x))),\] induced by inclusion, is zero.
\end{enumerate}
\end{defn}
As these properties are  quasi-isometry invariants,  the following theorem then demonstrates that finiteness properties are quasi-isometry invariants.
\begin{thm}[\cite{kapovich2005coarse},\cite{drutu2018geometric}]\label{thm:finitenesscoarse}
Let $G$ be a finitely generated group.\begin{enumerate}
\item $G$ is of type $F_n$ if and only if it is coarsely uniformly $(n-1)$-connected.
\item $G$ is of type $FP_n$ if and only if it is coarsely  uniformly $(n-1)$-acyclic.
\end{enumerate}
\end{thm}

The definition of coarse connectedness and acyclicity is somewhat difficult to work with as the property cannot be phrased in terms of a single Rips complex. However, in the case $n=1$, things are simpler:
\begin{lem}\label{lem:acyclic rips complex}
Suppose $X$ is a coarsely $1$-acyclic, discrete geodesic metric space. Then there exists an $r$ such that $P_r(X)$ is $1$-acyclic.
\end{lem}
\begin{proof} As $X$ is coarsely $1$-acyclic,  there is an $r\geq 1$ such that for $k=0,1$, the map $\tilde H_k(P_1(X))\rightarrow \tilde H_k(P_r(X))$, induced by inclusion, is zero. 
Let $\sigma=[x,y]$ be a 1-simplex in $P_r(X)$. There exists a discrete geodesic $x=x_0, x_1,\dots, x_n=y$ from $x$ to $y$ in $X$, where $n=d(x,y)\leq r$. Thus $d(x_i,x_j)= \lvert i-j\rvert \leq r$ for all $1\leq i,j\leq n$. It follows that for each $i$,  $\omega_i\coloneqq[x_0,x_{i-1},x_i]$ is a 2-simplex in $P_r(X)$, and so $\partial (\sum_{i=1}^n\omega_i)=\sum_{i=1}^n[x_{i-1},x_i]-\sigma$. It follows that every reduced $1$-cycle in $P_r(X)$ is homologous to one in $P_1(X)$, thus $P_r(X)$ is 1-acyclic.   
\end{proof}

We make use of the following class of spaces, defined in \cite{mosher2003quasi}, that  allow us to apply both topological and metric arguments.

\begin{defn}
A \emph{metric cell complex}  consists of the pair $(X,\bX)$, where $\bX$ is a connected cell complex,  $X$ is a bounded geometry metric space called the \emph{control space}, and there is a function $p$ from the set of cells of $\bX$ to  $X$ such that:
\begin{enumerate}
\item for each $d\in \mathbb{N}$, there is a number $N_d$ such that for every $x\in X$,  $p^{-1}(x)$ contains at most $N_d$ $d$-cells;
\item for each $d\in \mathbb{N}$, there is a number $M_d$ such that for every $d$-cell $e$ with attaching map $\phi_e$, 
$ \{p(e)\}\cup\{p(f)\mid f\cap \mathrm{im}(\phi_e)\neq \emptyset\}$ has diameter at most $M_d$;
\item  $p(\bX^{(0)})=X$.
\end{enumerate}
We say that a \emph{group $G$ acts on $(X,\bX)$} if $G$ acts on both $X$ and $\bX$ and the map $p$ is $G$-equivariant. Such an action is free if the corresponding action of $G$ is free on $X$.
\end{defn}

Almost all examples of metric cell complexes arising in this paper come from the following construction: 

\begin{exmp}Suppose a finitely generated group $G$, equipped with the word metric, acts freely and cocompactly on a cell complex $\bX$. We  define $p$ arbitrarily on a transversal for the set of $G$-orbits of cells and extend $p$ equivariantly. Thus $(G,\bX)$ is a metric cell complex admitting a free $G$ action. 
\end{exmp}

\begin{rem}\label{rem:metric complexes}
In fact, there is a much more general class of objects called \emph{metric complexes} that were  defined in \cite{kapovich2005coarse} and used extensively in \cite{margolis2018quasi}. For readability and conciseness, we will not define these objects here.  However, it will be clear to those familiar with the theory of metric complexes that all arguments in this article extend  to the setting of metric complexes.  \end{rem}

\begin{defn}
Let $(X,\bX)$ be a metric cell complex. Given a subset $A\subseteq Y$, we define $\bX[A]$ to be the largest subcomplex of $\bX$ such that $p(e)\in A$ for every cell $e$ in $\bX[A]$.  We say that $(X,\bX)$ is \emph{uniformly $n$-connected} if for all $r$, there exists an $s\geq r$ such that for all $x\in X$ and $k\geq n$, the map \[\pi_k(\bX[N_r(x)])\rightarrow \pi_k(\bX[N_s(x)]),\] induced by inclusion, is zero.
\end{defn}
 There is a corresponding notion of uniform $n$-acyclicity for metric cell complexes and more generally, metric  complexes.
\begin{prop}[{\cite[Proposition 3.20]{margolis2018quasi}}]\label{prop:coarse acyclicity gives complex}
Let $(X,d)$ be a bounded geometry metric space. Then:
\begin{enumerate}
\item $X$ is coarsely uniformly $n$-connected if and only if it is the control space of a uniformly $n$-connected metric cell complex.
\item $X$ is coarsely uniformly $n$-acyclic if and only if it is the control space of a uniformly $n$-acyclic metric  complex.
\end{enumerate}
Moreover, if $X=G$ is a finitely generated group equipped with the word metric, we can assume that the resulting metric cell complex in (1) and (2) admits a free $G$-action.
\end{prop}

Given a metric complex $(X,\bX)$, we define the \emph{support} of (cellular) chains and cochains as follows. If $\{\sigma_i\}_{i\in I}$ is the set of $k$-cells of $\bX$, then the support of a chain  $\sum_{i\in I} n_i \sigma_i\in C_k(\bX)$ is defined to be $\{p(\sigma_i)\mid n_i\neq 0\}$. Similarly, the support of a cochain $\alpha\in C^k(\bX)$ is defined to be $\{p(\sigma_i)\mid \alpha(\sigma_i)\neq 0\}$. We  define compactly supported cochains as cochains $\alpha$ with $\mathrm{supp}(\alpha)$ finite. Compactly supported cochains form a subcochain complex $C^\bullet_c(\bX)$ and we thus define the compactly supported cohomology $H^*_c(\bX)$ as the cohomology of this cochain complex.

In \cite{margolis2018quasi}, the author defines a coarse version of cohomology; see also \cite{roe1993coarse} and \cite{kapovich2005coarse}. The definition can be extended to coarsely uniformly $(n-1)$-acyclic metric spaces using the theory of metric complexes.
\begin{defn}
Let $X$ be a bounded geometry, coarsely uniformly $(n-1)$-connected metric space. Then for $k\leq n$, $H^k_\mathrm{coarse}(X)\coloneqq \ker (H^k_c(\bX)\rightarrow H^k(\bX))$, where $\bX$ is a uniformly $(n-1)$-connected metric complex with control space $X$.
\end{defn}
This is invariant under quasi-isometries. Moreover, a coarse embedding $A\rightarrow X$ induces a map $H^n_\mathrm{coarse}(X)\rightarrow H^n_\mathrm{coarse}(A)$. When   $G$ is a group of type $F_n$, $H^k_\mathrm{coarse}(G)=H^k(G,\mathbb{Z}_2G)$ for $k\leq n$ (recall all cohomology is taken with coefficients in $\bZ_2$).

We can now define coarse $PD_n$ spaces.
\begin{defn}
Let $(X,\bX)$ be a uniformly acyclic metric cell complex. We say that $\bX$ is a \emph{coarse $PD_n$ complex} if there exists a number $R$, chain maps \[C^{n-\bullet}_c(\bX)\xrightarrow{P} C_\bullet(\bX) \textrm{ and } C_\bullet(\bX) \xrightarrow{\overline P}C^{n-\bullet}_c(\bX),\] and chain homotopies $\bar{P}\circ P\overset{\Phi}{\simeq} \mathrm{id}_{C^{n-\bullet}_c(\bX)}$ and $P\circ \bar{P}\overset{\bar{\Phi}}{\simeq} \mathrm{id}_{C_\bullet(\bX)}$ such that:
\begin{enumerate}
\item for every  chain $\sigma$,  $\mathrm{supp}(\overline{P}(\sigma))$ and $\mathrm{supp}(\bar{\Phi}(\sigma))$ are contained in $N_R(\mathrm{supp}(\sigma))$;
\item for every cochain $\alpha$,  $\mathrm{supp}(P(\alpha))$ and $\mathrm{supp}(\Phi(\alpha))$ are contained in $N_R(\mathrm{supp}(\alpha))$.
\end{enumerate} 
A metric space is said to be \emph{coarse $PD_n$} if it is the control space of some coarse $PD_n$ complex. 
A \emph{coarse $PD_n$ group} is a finitely generated group that is a coarse $PD_n$ space.
\end{defn}
If $X$ is a coarse $PD_n$ space, then $H^n_\mathrm{coarse}(X)\cong \mathbb{Z}_2$ and $H^k_\mathrm{coarse}(X)=0$ for $k\neq n$. We refer to \cite{kapovich2005coarse} for a more thorough treatment of coarse $PD_n$ complexes and spaces. We recall the following useful fact about coarse $PD_n$ spaces, generalising an earlier result of Farb--Schwartz \cite{farbschwartz96}.
 \begin{thm}[Packing Theorem, {\cite[Part 4 of Theorem 7.7]{kapovich2005coarse}}]\label{thm:packing}
 If $W$ and $W'$ are coarse $PD_n$ subspaces of $X$ and $W\subseteq N_R(W')$ for some $R\geq 0$, then there exists an $R'\geq 0$ such that $ W'\subseteq N_{R'}(W)$.
 \end{thm}

We now define the technical hypothesis in the main theorem of \cite{margolis2018quasi}:
\begin{defn}
Let $X$ be a bounded geometry, coarsely uniformly $n$-acyclic metric space containing a coarse $PD_n$ subspace $W\subseteq X$. A coarse complementary component $C$ of $W$ is said to be \emph{essential} if the map $H^n_\mathrm{coarse}(C\cup W)\rightarrow H^n_\mathrm{coarse}(W)$, induced by inclusion, is zero.
\end{defn}
Every essential coarse complementary component is deep. Although the converse is not true, there are many instances in which deep and essential coarse complementary components are equivalent, e.g.  \cite[Proposition 1.2]{margolis2018quasi} and Proposition \ref{prop:deep imp essential}.
We refer to \cite{margolis2018quasi} for a thorough of discussion of essential components. 

In \cite{kapovich2005coarse}, Kapovich and Kleiner prove a coarse version of Jordan separation: a coarse $PD_{n}$ space $W$ coarsely separates a coarse $PD_{n+1}$ space into precisely two deep coarse complementary components. We say that such a component is a \emph{coarse $PD_{n+1}$ half-space with boundary $W$}. We will make use of the following criterion for showing coarse complementary components are essential.
\begin{lem}[{\cite[Corollary 6.7]{margolis2018quasi}}]\label{lem:essential if contains half-space}
Let $X$ be a bounded geometry, coarsely uniformly $n$-acyclic metric space containing a coarse $PD_n$ subspace $W\subseteq X$. If a coarse complementary component $C$ of $W$ contains a coarse $PD_{n+1}$ half-space with boundary $W,$ then $C$ is essential.
\end{lem}

We now state part of the main theorem of \cite{margolis2018quasi}:
\begin{thm}[\cite{margolis2018quasi}]\label{thm:essential 3-sep}
Let $G$ be a group of type $FP_{n+1}$. If $W\subseteq G$ is a coarse $PD_n$ subspace that coarsely separates $G$ into three essential coarse complementary components, then there is a subgroup $H\leq G$ with $d_\mathrm{Haus}(H,W)<\infty$.
\end{thm}

To prove Theorem \ref{thm:main intro}, we not only need Theorem \ref{thm:essential 3-sep}, but  some of the ingredients used to prove it. Specifically, we require the following lemma. 
\begin{lem}[{\cite[Proposition 5.19 and Lemma 6.13]{margolis2018quasi}}]\label{lem:cohom class and non-crossing}
Let $X$ be a bounded geometry metric space and let $(X,\bX)$ be a uniformly $n$-connected metric cell complex. Suppose $W\subseteq X$ is a coarse $PD_n$ subspace and that $C_1$, $C_2$ and $C_3$ are coarse disjoint, essential, coarse complementary components of $W$. Then there exist constants $D,E\geq 0$ such that:
\begin{enumerate}
\item there exist non-zero, distinct cohomology classes \[[\alpha_1],[\alpha_2]\in \ker (H^{n+1}_c(\bX)\rightarrow H^{n+1}(\bX))\] such that for every $x\in W$, $[\alpha_i]$ can be supported by a cocycle supported in $N_D(x)$ for $i=1,2$;
\item  $\alpha_i$ can be represented by a cocycle supported in $C_i\setminus N_E(W)$ for $i=1,2$;
\item if $[\omega]\in \ker (H^{n+1}_c(\bX)\rightarrow H^{n+1}(\bX))$ is a cohomology class represented by cocycles supported in any two of $C_1\backslash N_E(W)$, $C_2\backslash N_E(W)$ and $C_3\backslash N_E(W)$, then $[\omega]=0$.  
\end{enumerate}
\end{lem}

%% file: bundles.tex
\section{Coarse bundles}\label{sec:coarse bundles}

In this section we introduce coarse bundles, one of the  key concepts of this article. Similar definitions appear in \cite{farbmosher2000abelianbycyclic},  \cite{kapovich2005coarse} and \cite{whyte2010coarse}. 

\begin{defn}\label{defn:coarse bundle}
Let $(X,d)$, $(F,d_F)$ and $(B,d_B)$ be bounded geometry, quasi-geodesic metric spaces. We say that \emph{$X$ is a coarse bundle over $B$ with fibre $F$} if there are constants $K\geq 1$, $A,E\geq 0$, proper non-decreasing functions $\eta,\phi:\bR_{\geq 0}\rightarrow \bR_{\geq 0}$ and a map $p:X\rightarrow B$ such that:
\begin{enumerate}
\item \label{defn:coarse bundle 1} $p$ is $(K,A)$-coarse Lipschitz, i.e. for all $x,x'\in X$, \[d_B(p(x),p(x'))\leq Kd(x,x')+A;\]
\item \label{defn:coarse bundle 2}for all $b\in B$, $D_b\coloneqq p^{-1}(N_E(b))$ is known as a \emph{fibre} of $X$ and there is an $(\eta,\phi)$-coarse embedding $s_b:F\rightarrow X$ with \[\im (s_b)\subseteq D_b\subseteq N_A(\im (s_b));\]
\item \label{defn:coarse bundle 3}for all $b,b'\in B$,  $d_\Haus(D_b,D_{b'})\leq K d_B(b,b')+A$.
\end{enumerate}
When the above holds, we say that $p:X\rightarrow B$ is a \emph{$(K,A,E)$-coarse bundle}, or simply a \emph{coarse bundle}. We call $X$ the \emph{total space}, $B$  the \emph{base space} and $F$  the \emph{fibre}. 
\end{defn}

\begin{rem}
Given a coarse bundle $p:X\rightarrow B$ and a fibre $D_b$ as above, the space $F$ can be recovered from $D_b$ up to quasi-isometry using Lemma \ref{lem:coarse embedding}. Thus there is no ambiguity in referring to $D_b$ as a fibre, as it is just a distorted copy of $F$ sitting inside $X$.
\end{rem}

\begin{prop}\label{prop:coarse bundle QI}
Suppose that $p:X\rightarrow B$ is a coarse bundle with fibre $F$ and that $f:Y\rightarrow X$ is a quasi-isometry. Then $p\circ f:Y\rightarrow B$ is also a coarse bundle with fibre $F$.
\end{prop}
Proposition \ref{prop:coarse bundle QI} can be deduced from Lemma \ref{lem:technical pullback}.
Most of the coarse bundles that we are interested in satisfy the following:
\begin{defn}
Let $p:X\rightarrow B$ be a coarse bundle. A quasi-isometry $f:X\rightarrow X$ is said to be \emph{$A$-fibre-preserving} if for every $b\in B$, there exists a $b'\in B$ such that $d_\Haus(f(D_b),D_{b'})\leq A$.
We say that $p:X\rightarrow B$ is  \emph{quasi-homogeneous} if there exist constants $K\geq 1$ and $A \geq 0$ such that for all $x,y\in X$, there is an $A$-fibre-preserving $(K,A)$-quasi-isometry $f:X\rightarrow X$ such that 
$d(f(x),y)\leq  A$.
\end{defn}
Quasi-homogeneity of coarse bundles is easily seen to be a quasi-isometry invariant.
The following lemma follows from the  preceding definitions.
\begin{lem}\label{lem:coarse inverse fibre-preserving}
Let $p:X\rightarrow B$ be a quasi-homogeneous coarse bundle and $x_0\in X$ be a basepoint. Then there exist constants $K\geq 1$ and $A\geq 0$ such that for every $x\in X$:
\begin{enumerate}
\item There is an $A$-fibre-preserving $(K,A)$-quasi-isometry $f_x:X\rightarrow X$ such that $d(f_x(x),x_0)\leq A$;
\item $f_x$ has a coarse inverse $\overline{f_x}:X\rightarrow X$ which is also an $A$-fibre-preserving $(K,A)$-quasi-isometry such that $\overline{f_x}\circ f_x$ and $f_x\circ\overline{f_x}$ are $A$-close to the identity.
\end{enumerate}
\end{lem}

Arbitrary quasi-isometries between coarse bundles don't preserve the coarse bundle structure. However, if two coarse bundles are quasi-isometric via a fibre-preserving quasi-isometry, then it follows from Proposition \ref{prop:fibreqi induces qi between fibres} that the fibres are quasi-isometric. Moreover, the following lemma demonstrates that the base spaces are quasi-isometric:

\begin{lem}\label{lem:fibre-preserving qi imp qi between base spaces}
Let $p:X\rightarrow B$ and $p':Y\rightarrow B'$ be coarse bundles. Suppose $f:X\rightarrow Y$ is a $(K,A)$-quasi-isometry that is $A$-fibre-preserving.  Then $f$ induces a $(K',A')$-quasi-isometry $\hat f:B\rightarrow B'$ such that $d_\mathrm{Haus}(D_{\hat f (b)},f(D_b))\leq A$. The constants $K'$ and $A'$ depend only on $K$ and $A$.
\end{lem}
\begin{proof}
By hypothesis, such a function $\hat f$ exists and is coarsely well-defined. We need only  show that $\hat f$ is a quasi-isometry. By increasing $K$ and $A$ if necessary, we may assume that $p:X\rightarrow B$ and $p':X'\rightarrow B'$ are $(K,A,E)$-coarse bundles. We fix $b,b'\in B$ and  pick $x\in D_b$ and $x'\in D_{b'}$ such that $d(x,x')\leq d_\Haus(D_b,D_{b'})\leq Kd_B(b,b')+A$. We also observe that \[d_B(b,b')\leq d_B(p(x),p(x'))+2E\leq Kd(x,x')+A+2E.\]

It follows from the definition of $\hat f$ that there is a $y\in D_{\hat f(b)}$ with  $d(y,f(x))\leq A$. Thus $d_{B'}(\hat f(b),p'(f(x)))\leq KA+A+E$. The same argument shows that $d_{B'}(\hat f(b'),p'(f(x')))\leq KA+A+E$, and so 
\begin{align*}
d_{B'}(\hat f(b),\hat f(b'))&\leq d(p'(f(x)),p'(f(x')))+2(KA+A+E)\\
&\leq Kd(f(x),f(x'))+2KA+3A+2E\\
&\leq K^2d(x,x')+ 3KA+3A+2E\\
&\leq K^3d_B(b,b')+ K^2A+ 3KA+3A+2E.\end{align*} 

We observe that $d_\Haus(D_b,D_{b'})\leq Kd_\Haus(f(D_b),f(D_{b'}))+KA$ and so
 \begin{align*}
	d_B(b,b')	&\leq	Kd_X(x,x')+A+2E\\
				&\leq	Kd_\Haus(D_b,D_{b'})+A+2E\\
				&\leq	K^2d_\Haus(f(D_b),f(D_{b'}))+K^2A+A+2E\\
				&\leq	K^2d_\Haus(D_{\widehat{f}(b)},D_{\widehat{f}(b')})+3K^2A+A+2E\\
				&\leq	K^3d_{B'}(\widehat{f}(b),\widehat{f}(b'))+4K^2A+A+2E
.\end{align*} 

Finally, for each $b'\in B'$ there is some $z'\in D_{b'}$ and $z\in X$ with $d(f(z),z')\leq A$. Let $b\coloneqq p(z)$. Since $z\in D_{b}$ and  $d_\Haus(D_{\hat f(b)},f(D_{b}))\leq A$,  there is some $y\in D_{\hat f(b)}$ with $d(y,f(z))\leq A$. Thus $d(z',y)\leq 2A$ and so \[d_{B'}(b',\hat f(b))\leq d_{B'}(p'(z'),p'(y))+2E\leq 2KA+A+2E.\] Thus $\hat f$ is indeed a quasi-isometry.
\end{proof}

\begin{defn}
Let $p:X\rightarrow B$ be a coarse bundle, let $W$ be a bounded geometry quasi-geodesic metric space,  and let $\lambda:W\rightarrow B$ be a coarse embedding. Then the \emph{pullback bundle} is a coarse bundle $p_{W}:X_{W}\rightarrow W$ equipped with a coarse embedding $\widehat \lambda :X_W\rightarrow X$ and a constant $R\geq 0$  such that for every $w\in W$, $d_\Haus(\widehat \lambda(D_w),D_{\lambda(w)})\leq R$. 

\end{defn}
\begin{prop}\label{prop:pullbacks exist}
Let $p:X\rightarrow B$,  $W$ and $\lambda$ be as above.  Then the pullback bundle $p_{W}:X_{W}\rightarrow W$ always exists and is unique up to quasi-isometry
\end{prop}

To prove this, we require the following lemma. 
\begin{lem}\label{lem:pullbacks exist}
Suppose $p:X\rightarrow B$ is a $(K,A,E)$-coarse bundle and $\lambda:W\rightarrow B$ is a coarse embedding. Then there exists a quasi-geodesic metric space $X_W$ and a coarse embedding $\hat \lambda:X_W\rightarrow B$ such that \[d_\Haus(p^{-1}(N_{E}(\im(\lambda))),\im(\hat \lambda))<\infty.\]
\end{lem}
\begin{proof}
We recall that a subset of a metric space is said to be coarsely embedded if it satisfies one of the conditions of Lemma \ref{lem:coarse embedding}. By Lemma \ref{lem:coarse embedding},  it is sufficient to show that $Y\coloneqq p^{-1}(N_{E}(\im(\lambda)))$ is coarsely embedded.

 It follows from Lemma \ref{lem:coarse embedding} and Definition \ref{defn:coarse bundle} that there exists a $t>0$ and a proper non-decreasing function $\eta:\bR_{\geq 0} \rightarrow \bR_{\geq 0}$ such that for each $b\in B$, every $x,y\in D_b$ can be joined by a $t$-chain in $D_b$ of length at most $\eta(d(x,y))$.
Since $\im(\lambda)$  is coarsely embedded, we can take $t$ and $\eta$ sufficiently large so that every $x,y\in \im(\lambda)$ can be joined by a $t$-chain in $\im(\lambda)$ of length at most $\eta(d(x,y))$.

Suppose that $p:X\rightarrow B$ is a $(K,A,E)$-coarse bundle. Pick $x,y\in Y$ and $b_x,b_y\in \im(\lambda)$ such that $d(b_x,p(x)),d(b_y,p(y))\leq E$. Thus $d_B(b_x,b_y)\leq Kd(x,y)+A+2E$, and so $b_x$ and $b_y$ can be joined by a $t$-chain in $\im(\lambda)$ of length at most $\eta(Kd(x,y)+A+2E) \eqqcolon M$.  Notice that if $d(b,b')\leq t$, then $d_\Haus(D_b,D_{b'})\leq Kt+A$. We can thus ``lift'' this $t$-chain in $\im(\lambda)$ to a $(Kt+A)$-chain in $Y$ from $x$ to some $z\in D_{b_y}$ of length at most $M$.

Now observe that $z,y\in D_{b_y}$ and that $d(z,y)\leq d(x,y)+M(Kt+A)$. Therefore $z$ and $y$ can be joined by a $t$-chain in $D_{b_y}\subseteq Y$ of length at most $\eta(d(x,y)+M(Kt+A))$. Thus $x$ and $y$ can be joined by a $(Kt+A)$-chain in $Y$ of length at most \[N\coloneqq M+ \eta(d(x,y)+M(Kt+A)).\] The result now follows since $N$ can be expressed as a proper non-decreasing function of $d(x,y)$.
\end{proof}

Proposition \ref{prop:pullbacks exist} follows from Lemma \ref{lem:pullbacks exist} and the following lemma:
\begin{lem}\label{lem:technical pullback}
Let $p:X\rightarrow B$ be a $(K,A,E)$-coarse bundle, let $W$ and $X_W$ be quasi-geodesic metric spaces, and let $\lambda:W\rightarrow B$ and $\widehat \lambda:X_W\rightarrow X$ be coarse embeddings such that \[d_\Haus(\im(\widehat{\lambda}),p^{-1}(N_E(\im(\lambda)))<\infty.\] Then there are constants $K'\geq 1$, $A',E'\geq 0$ and a $(K',A',E')$-coarse bundle $p_W:X_W\rightarrow W$   such that \[d_\Haus(\widehat \lambda(D_w),D_{\lambda(w)})\leq A'\] for every $w\in W$.
\end{lem}
\begin{proof}
We may choose $K\geq 1$, $t,A,E\geq 0$ and proper non-decreasing functions $\eta$ and $\phi$ such that:
\begin{itemize}
\item $\lambda:W\rightarrow B$ is an $(\eta,\phi)$-coarse embedding.
\item $p:X\rightarrow B$ is a $(K,A,E)$-coarse bundle.
\item Every $x,y\in W$ can be joined by a $t$-chain of length at most $Kd_W(x,y)+A$.
\item There is a quasi-geodesic metric space $X_W$ and an $(\eta,\phi)$-coarse embedding $\hat \lambda:X_W\rightarrow X$ such that $d_\Haus(p^{-1}(N_{E}(\im(\lambda))),\im(\hat \lambda))\leq A$. The existence of $X_W$ and $\hat\lambda$ follows from Lemma \ref{lem:pullbacks exist}.
\item Every $x,y\in X_W$ can be joined by a $t$-chain of length at most $Kd_{X_W}(x,y)+A$.
\end{itemize} 

We  define $p_W:X_W\rightarrow W$ such that  for every $x\in X_W$, 
\[d((\lambda\circ p_W)(x),(p\circ \widehat \lambda)(x))\leq KA+A+E.\] The existence of such a $p_W$ follows from our choice of $\widehat \lambda$. We show that $p_W:X_W\rightarrow W$ is a coarse pullback of $\lambda$.

We define $E'\coloneqq \widetilde \eta(2KA+2A+2E)$ and claim that $p_W:X_W\rightarrow W$ is a $(K',A',E')$-coarse bundle for some $K'\geq 1$ and $A'\geq 0$ to be determined.
 We first show that $p_W$ is coarse Lipschitz. For any  $x,y\in X_W$, we have  
\begin{align*}
\eta\big(d_W(p_W(x),p_W(y))\big)&\leq d_B\big((\lambda\circ p_W)(x),(\lambda\circ p_W)(y)\big)\\
&\leq d_B\big((p\circ \widehat \lambda)(x),(p\circ \widehat \lambda)(y)\big)+2(KA+A+E)\\
&\leq Kd_X(\widehat\lambda(x),\widehat \lambda(y))+2KA+3A+2E\\
&\leq K\phi(d_{X_W}(x,y))+2KA+3A+2E.
\end{align*}
Thus there is a number $M$ such that $d_W(p_W(x),p_W(y))\leq M$ whenever $d_{X_W}(x,y)\leq t$. Recall that any $x,y\in X_W$ can be joined by a $t$-chain of length at most $Kd_{X_W}(x,y)+A$. Thus  $p_W(x)$ and $p_W(y)$ can be joined by an $M$-chain of length $Kd_{X_W}(x,y)+A$, and so  $d_W(p_W(x),p_W(y))\leq M(Kd_{X_W}(x,y)+A)$ for any $x,y\in X_W$, verifying that $p_W$ is coarse Lipschitz.

As in Definition \ref{defn:coarse bundle}, we define the fibres of $p_W:X_W\rightarrow W$ to be $D^W_w\coloneqq p_W^{-1}(N_{E'}(w))$ for each $w\in W$.  We pick $w\in W$ and observe that $\im(s_{\lambda(w)})\subseteq D_{\lambda(w)}\subseteq p^{-1}(N_E(\im(\lambda)))$, where $s_{\lambda(w)}$ is in Definition \ref{defn:coarse bundle}. We thus define a coarse embedding $s_w:F\rightarrow X_W$ such that \[d((\widehat \lambda\circ s_w)(f),s_{\lambda(w)}(f))\leq A\] for all $f\in F$. It is straightforward to verify  that for some $A'$ sufficiently large,
 \begin{align*}\im(s_w)\subseteq D_w^W\subseteq N_{A'}(\im(s_w))\end{align*} for all $w\in W$.
The remainder of the proof follows easily from the definitions.
\end{proof}

Our main examples of coarse bundles are finitely generated groups containing almost normal subgroups. Many  basic facts concerning almost normal subgroups are proven in \cite{connermihalik2014}. For instance:
\begin{prop}\label{prop:alnorm is comm inv}
Suppose $G$ is a group and $H\alnorm G$. If $H'\leq G$ is commensurable to $H$, then $H'\alnorm G$.
\end{prop}
\begin{proof}
It is sufficient to prove the lemma in the case that $H$ is a finite index subgroup of $H'$ or vice-versa, which is proved in Lemma 3.10 of \cite{connermihalik2014}.
\end{proof}

 We  frequently make use of the following characterisation of almost normal subgroups, also shown in \cite{vavrichek2013commensurizer} and  \cite{connermihalik2014}.
\begin{prop}\label{prop:coset characterisation}
Let $G$ be a finitely generated group with $H\leq G$. The following are equivalent:
\begin{enumerate}
\item $H\alnorm G$;
\item for every $g\in G$, the coset $gH$ is at finite Hausdorff distance from $H$.
\end{enumerate}
\end{prop}
\begin{proof}
It is shown in \cite[Corollary 2.14]{mosher2011quasiactions} that two subgroups $H$ and $K$ of $G$ are commensurable if and only if they are at finite Hausdorff distance. It is easy to see that for all $g\in G$, $gHg^{-1}$ and $gH$ are at finite Hausdorff distance. Thus $gH$ and $H$ are at finite Hausdorff distance if and only if $H$ and $gHg^{-1}$ are commensurable.
\end{proof}
This characterisation allows us to define the quotient space $G/H$.
\begin{defn}
Given a finitely generated group $G$ and an almost normal subgroup $H\alnorm G$, the \emph{quotient space} $G/H$ is the set of left cosets of $H$, equipped with the metric $d_{G/H}(gH,g'H):=d_\mathrm{Haus}(gH,g'H)$.
\end{defn}
The natural left action of $G$ on $G/H$ is an isometric action. 

\begin{prop}\label{prop:almost normal subgroups coarse bundles}
Let $G$ be a finitely generated group with $H\alnorm G$ finitely generated and let $p:G\rightarrow G/H$ be the \emph{quotient map} $g\mapsto gH$. Then $G$ is a quasi-homogeneous coarse bundle over $G/H$ with fibre $H$.
\end{prop}
\begin{proof}
We  fix a generating set $S$ of $G$ and let $d$ be the corresponding word metric on $G$. We set $K\coloneqq \max_{s\in S}d_\Haus(H,sH)$ and claim that $p:G\rightarrow G/H$ is a $(K,0,0)$-coarse bundle. We first note that  $G/H$ has bounded geometry. This follows from the fact that $G$ has bounded geometry and distinct  cosets are disjoint. 

We use Proposition \ref{prop:quasi geodesic} to show that the quotient space $G/H$ is  quasi-geodesic. 
If $gH,kH\in G/H$, then there is an $h\in H$ with  $d(g,kh)\leq d_\Haus(gH,kH)$. Since $S$ is a generating set, $g^{-1}kh=s_1\dots s_t$ where $t\leq d_\Haus(gH,kH)$ and for each $i$, either $s_i$ or its inverse is contained in $S$. 
Then for every $i$  \[d_\mathrm{Haus}(gs_1\dots s_{i} H, gs_1\dots s_is_{i+1} H)=d_\mathrm{Haus}( H, s_{i+1} H)\leq K,\] demonstrating that $gH$ and $kH$ can be joined by a $K$-chain of length at most $d_\Haus(gH,kH)$.

To prove (\ref{defn:coarse bundle 1}) of Definition \ref{defn:coarse bundle}, we pick $g, k\in G$ and we note the above argument shows $gH$ and $kH$ can be joined by a $K$-chain in $G/H$ of length at most $d(g,k)$. Thus $d_{G/H}(p(g),p(k))=d_\mathrm{Haus}(gH,kH)\leq Kd(g,k)$.
Condition (\ref{defn:coarse bundle 2}) of Definition \ref{defn:coarse bundle} follows easily from Lemma \ref{lem:sbgpis coarsely embedding} and the fact that left multiplication in $(G,d)$ is isometric. The metric defined on $G/H$ ensures that (\ref{defn:coarse bundle 3}) is automatically true. Quasi-homogeneity of the coarse bundle is also evident, since the transitive action of $G$ on itself permutes the set of left $H$-cosets.
\end{proof}

 Another construction of the quotient space is carried out in \cite{kronmoller08roughcayley}:
\begin{defn}
Let $G$ be a group and $H$ be an almost normal subgroup such that $G$ is finitely generated relative to $H$, i.e. there exists a finite set $S\subseteq G$ such that $S\cup H$ generates $G$. Then the \emph{rough Cayley graph} is a graph with vertex set $G/H$ and edge set $\{(gH,gsH)\mid s\in S, g\in G\}$. 
\end{defn}
 It is shown in \cite{kronmoller08roughcayley} that the rough Cayley graph is locally finite and is well-defined up to quasi-isometry. The main focus of \cite{kronmoller08roughcayley} is the situation in which $G$ is a totally disconnected locally compact group and  $H$ is a compact open subgroup; in this case the rough Cayley graph is frequently referred to as the \emph{Cayley--Abels} graph. In the case where $G$ is finitely generated, it is not hard to see that the quotient space $G/H$ is equivariantly quasi-isometric to the rough Cayley graph of $G$ with respect to $H$, see also \cite{connermihalik2014}.

\begin{lem}\label{lem:coarse bundle ccc}
Let $X$ and  $B$ be discrete geodesic metric spaces and let $p:X\rightarrow B$ be a $(K,A,E)$-coarse bundle. Pick a basepoint $b_0\in B$. Then:
\begin{enumerate}
\item\label{lem:coarse bundle ccc1} 
If $C$ is a coarse complementary component of $D_{b_0}$, then $C'\coloneqq N_E(p(C))$ is a coarse complementary component of $b_0$. Moreover,  $d_\Haus(p^{-1}(C'),C)<\infty$.
\item \label{lem:coarse bundle ccc2} 
If $C'\subseteq B$ is a coarse complementary component of $b_0$, then $C\coloneqq p^{-1}(C')$ is a coarse complementary component of $D_{b_0}$.
\end{enumerate}In both cases, $C$ is deep if and only if $C'$ is deep.
\end{lem}
\begin{proof}
(\ref{lem:coarse bundle ccc1}): As $C$ is a coarse complementary component of $D_{b_0}$, there is an $r$ such that $\partial C\subseteq N_{r}(D_{b_0})$. Suppose $y\in \partial C'$ with $x\in C'$ and $d(x,y)=1$. We may choose $\tilde x\in C$ such that $d_B(p(\tilde x),x)\leq E$. Since $d_B(x,y)=1$, we see that $d_\Haus(D_x,D_y)\leq K+A$, and so there exists some $\tilde{y}\in D_y$ with $d(\tilde x, \tilde y)\leq K+A$.  Since $\tilde x\in C$ and $\tilde y\notin C$, we see that $\tilde{y}\in N_{K+A}(\partial C)\subseteq N_{K+A+r}(D_{b_0})$. Thus $d(y,b_0)\leq K(K+A+r)+A+2E$  so that $C'$ is a coarse complementary component of $b_0$.

(\ref{lem:coarse bundle ccc2}): As $C'$ is a coarse complementary component of $b_0$, there is an $r$ such that $\partial C'\subseteq N_{r}(b_0)$. Suppose $y\in \partial C$ with $x\in C$ and $d(x,y)=1$. Then $p(x)\in C'$, $p(y)\notin C'$ and $d(p(x),p(y))\leq K+A$. It follows that $d(p(y),b_0)\leq K+A+r$ and so $d_\Haus(D_{p(y)},D_{b_0})\leq K(K+A+r)+A$. Thus $d(y,D_{b_0})\leq K(K+A+r)+A$ so that $C$ is a coarse complementary component of $D_{b_0}$.

In both cases,  $C$ is deep if and only if $C'$ is.
\end{proof}
\begin{cor}\label{cor:ends}
Suppose $X$ and $B$ are discrete geodesic metric spaces and that $p:X\rightarrow B$ is a coarse bundle. Let $b_0\in B$. Then $D_{b_0}$ coarsely $n$-separates $X$ if and only if $B$ has at least $n$ ends. In particular, if $G$ is a finitely generated group containing a finitely generated subgroup $H\alnorm G$, then $\tilde e(G,H)=e(G/H)$.
\end{cor}
 \begin{proof}
If $D_{b_0}$ coarsely $n$-separates $X$, we can choose $n$ deep, coarse disjoint, coarse complementary components $C_1,\dots, C_n$ of $D_{b_0}$. Lemma \ref{lem:coarse bundle ccc} ensures there is an $E$ such that each $C'_i\coloneqq N_E(p(C_i))$ is a deep  coarse complementary component of $b_0$. Moreover, for $i\neq j$ we see that $C'_i$ and $C'_j$ are coarse disjoint since $C_i$ and $C_j$ are. Thus $B$ has at least $n$ ends. A similar argument demonstrates that $D_{b_0}$ coarsely $n$-separates $X$ if $B$ has at least $n$ ends.
We recall from Proposition \ref{prop:relends vs cccs} that $H$ coarsely $n$-separates $G$ if and only $\tilde{e}(G,H)\geq n$. It follows that $\tilde{e}(G,H)=e(G/H)$.
\end{proof} 

The following corollary is essentially due to Hopf \cite{hopf1944enden}, since any connected locally finite vertex-transitive graph has exactly $0$, $1$, $2$ or infinitely many ends and $G/H$ is quasi-isometric to a vertex-transitive graph, namely the rough Cayley graph. See also \cite{kropholler1989relative} and \cite{connermihalik2014}.
\begin{cor}\label{cor:ends hopf}
Let $G$ be a finitely generated group and $H\alnorm G$ be an almost normal subgroup. Then $e(G/H)=\tilde e(G,H)=0,1,2$ or $\infty$.
\end{cor}
This is no longer true if we drop the  assumption that $H\alnorm G$, since then $\tilde e(G,H)$ can take any value in $\bN\cup \infty$; see \cite{kropholler1989relative}. 
Corollary \ref{cor:ends hopf} tells us that the $e(G/H)\geq 3$ hypothesis in Theorem \ref{thm:main intro} is equivalent to the a priori stronger hypothesis $e(G/H)=\infty$.

We now relate coarse complementary components to group actions on trees. A \emph{minimal} action of a group $G$ on a tree $T$ is one in which no proper subtree of $T$ is fixed by $G$.  Given an  edge $e\in ET$ with midpoint $m_e$, the components of $T\backslash m_e$ are called \emph{halfspaces based at $e$}. 
\begin{lem}\label{lem:halfspace vs ccc}
Suppose $G$ is a finitely generated group which acts minimally on a tree $T$. Let $v_0\in VT$ be a basepoint. If $e\in ET$ is an edge with stabiliser $H$ and $\mathfrak{h}$ is a halfspace based at $e$, then \[C_\mathfrak{h}\coloneqq \{g\in G\mid gv_0\in \mathfrak{h}\}\] is a deep coarse complementary component of $H$.
\end{lem} 
\begin{proof}
It is a well known fact that $C_\mathfrak{h}$ is an $H$-almost invariant subset, see for instance \cite[Lemma 1.8]{scott1977ends}. Thus Proposition \ref{prop:relends vs cccs} tells us that $C_\mathfrak{h}$ is a deep coarse complementary component of $H$.
\end{proof}
 
\begin{cor}\label{cor:ends tree vs space}
Let $G$ be a finitely generated group and $H\alnorm G$. Suppose $G$ acts minimally on a  tree $T$, all of whose edge stabilisers are commensurable to $H$. Then $e(T)\leq \tilde e(G,H)$.
\end{cor}
\begin{proof}
We fix a basepoint $v_0\in V$.  If $e(T)\geq n$, we choose  halfspaces $\mathfrak{h}_1,\mathfrak{h}_2,\dots, \mathfrak{h}_n$  that are pairwise disjoint. Since all edge stabilisers are commensurable to $H$,  $C_{\mathfrak{h}_1},C_{\mathfrak{h}_2},\dots, C_{\mathfrak{h}_n}$ are a disjoint collection of deep coarse complementary components of $H$. Thus  $H$ coarsely $n$-separates $G$ and so Proposition \ref{prop:relends vs cccs} ensures  $\tilde e(G,H)\geq n$.
\end{proof}

It is well-known that the quotient of a finitely presented group by a finitely generated group is finitely presented. The following proposition is a geometric analogue of this fact.
\begin{prop}\label{prop:base space acyclic}
Let $G$ be a group of type $FP_2$ with $H\alnorm G$ finitely generated. Then $G/H$ is coarsely 1-acyclic.
\end{prop}
\begin{lem}\label{lem:induced to base}
For each $r\geq 0$, there is a $t\geq r$ and a simplicial map \[\phi:P_r(G)\rightarrow P_t(G/H)\] defined by $[g_0,\dots,g_n]\mapsto [g_0H,\dots,g_nH]$.
\end{lem}
\begin{proof}
By Proposition \ref{prop:almost normal subgroups coarse bundles}, the quotient map $p:G\rightarrow G/H$ is a $(K,A,E)$-coarse bundle for some $K$, $A$ and $E$. Thus whenever $[g_0,\dots,g_n]$ is a simplex of $P_r(G)$, \[d_\Haus(g_iH,g_jH)\leq Kd(g_i,g_j)+A\leq Kr+A\] for all $i,j$,  and so $[g_0H,\dots,g_nH]$ is a simplex of $P_{Kr+A}(G/H)$.
\end{proof}

\begin{proof}[Proof of Proposition \ref{prop:base space acyclic}]
By Theorem \ref{thm:finitenesscoarse} and Lemma \ref{lem:acyclic rips complex}, we can  pick $r$ sufficiently large such that $P_r(H)$ and $P_r(G/H)$ are connected and $P_r(G)$ is 1-acyclic. We pick $t\geq r$ so that Lemma \ref{lem:induced to base} holds. It is sufficient to show that every loop in $P_r(G/H)$ of the form \[\sigma= (x_0H,x_1H, \dots,x_{n-1}H,x_0H)\] is the boundary of some 2-chain in $P_t(G/H)$. We do this by lifting $\sigma$ to a cycle $\tilde \sigma$ in $P_r(G)$. We set $\tilde{x}_0\coloneqq x_0$. We inductively define $\tilde{x}_i$ as follows: since $d_\mathrm{Haus}(x_{i-1}H,x_{i}H)\leq r$, there is an $\tilde{x}_i\in x_iH$ with $d(\tilde{x}_{i-1},\tilde{x}_{i})\leq r$. 
We thus get a path $\tilde{x}_0,\dots, \tilde{x}_n$ in $P_r(G)$, where $\tilde{x}_0H=\tilde{x}_nH=x_0H$. 
Since $P_r(H)$ is connected, we choose a path $(\tilde{x}_n,\tilde{x}_{n+1},\dots, \tilde{x}_m=\tilde{x}_0)$ in $P_r(x_0H)$. Thus $\tilde{\sigma}:=(\tilde{x}_0,\dots, \tilde{x}_m)$ is a loop in $P_r(G)$ such that $\phi_{\#}(\tilde{\sigma})=\sigma$, with $\phi$  as in Lemma \ref{lem:induced to base}. 
By assumption, there is an $\omega\in C_2(P_r(G))$ such that $\partial\omega=\tilde{\sigma}$, and hence $\partial\phi_{\#}(\omega)=\phi_{\#}(\tilde \sigma)=\sigma$ as required.
\end{proof}

We can now prove Theorem \ref{thm:gog}, which is a relative version of Dunwoody accessibility. We recall the graph theoretic interpretation of Dunwoody accessibility, first noted in \cite{mosher2003quasi}. See also \cite{drutu2018geometric}.
\begin{thm}[\cite{dunwoody1985accessibility}]\label{thm:dunwoody accessibility}
Let $X$ be a 2-dimensional, locally finite, 1-acyclic simplicial complex with $\mathrm{Aut}(X)$ acting cocompactly on $X$. There exists a collection of disjoint compact tracks $\{\tau_i\}$ that is $\mathrm{Aut}(X)$-invariant, has finitely many $\mathrm{Aut}(X)$ orbits, and every component of $X\backslash \cup_i{\tau_i}$ has at most one end. 
\end{thm}

\reldunaccessibility
\begin{proof}
As in Remark \ref{rem:remetrise}, we can remetrise $G/H$ to be a discrete geodesic metric space, which we denote $Y$.
It follows from Proposition \ref{prop:base space acyclic} that $Y$ is a coarsely 1-acyclic. Thus Lemma \ref{lem:acyclic rips complex} ensures that $P_s(Y)$ is  $1$-acyclic for some sufficiently large $s$.   Let $X\coloneqq P_s(Y)^{(2)}$. By applying Theorem \ref{thm:dunwoody accessibility} to $X$, we see that there exists a collection of finite disjoint tracks $\{\tau_i\}$, containing finitely many $G$-orbits, such that every component of $X\setminus \cup \{\tau_i\}$ has at most 1 end.

Let $T$ be the dual tree to $\{\tau_i\}$. By subdividing edges if necessary, we may assume that $G \curvearrowright T$ without edge inversions. We claim that the stabiliser of a track $\tau_i$, hence an edge of $T$, is commensurable to $H$. Indeed, let $C\subseteq X^{(0)}$ be the set of finitely many vertices  that are contained in a simplex which intersects $\tau$. The pointwise stabiliser of $C$ is commensurable to $H$ and stabilises $\tau_i$. Moreover, the setwise stabiliser of $C$ contains the stabiliser of $\tau_i$ and is contained in  finitely many left $H$-cosets as the track $\tau_i$ is finite and $X$ is locally finite. Thus the stabiliser of $\tau_i$ is commensurable to $H$.

Suppose $C$ is a  component of $X\setminus \cup \{\tau_i\}$ corresponding to some $v\in VT$. Let $X_v$ denote the subcomplex of $X$ consisting of the closure of cells that intersect $C$.  As $X_v$ is connected and  $G_v$ acts cocompactly on $X_v$ with  cell stabilisers commensurable to $H$, hence finitely generated, it follows that $G_v$ is finitely generated.

Note that $C$ has finite Hausdorff distance from $Y_v\coloneqq \{gH\mid g\in G_v\}\subseteq X$.  
Observe that the vertex stabiliser $G_v$ contains some adjacent edge group, i.e. a subgroup commensurable to $H$. Thus $H_v\coloneqq H\cap G_v$ is a finite index subgroup of $H$ so that $H_v\alnorm G_v$. The map  $\phi:G/H_v\rightarrow G/H=X^{(0)}$ defined by $gH_v\mapsto gH$ is a quasi-isometry. Now observe that $\phi^{-1}(Y_v)=\{gH_v\mid g\in G_v \}$. Since $C$ has at most one end, the quotient space $G_v/H_v$ has at most one end, hence by  Corollaries \ref{cor:ends} and \ref{cor:ends tree vs space}, $G_v$ cannot split over any subgroup commensurable to $H$.
\end{proof}

%% file: maintheorem.tex
\section{QI rigidity of groups containing almost normal subgroups}\label{sec:mainthm}

We are now in a position to prove Theorem \ref{thm:main intro}. We actually prove a more general theorem.

\begin{thm}\label{thm:main fib bundle}
Let $p:X\rightarrow B$ be a quasi-homogeneous coarse fibre bundle with fibre $F$ such that:
\begin{enumerate}
\item $X$ is  a coarsely uniformly $n$-acyclic metric space;
\item $B$  has at least three ends;
\item $F$ is a coarse $PD_n$ space.
\end{enumerate}

Suppose $G$ is a finitely generated group quasi-isometric to $X$ and that $K\geq 1$ and $A\geq 0$. Then there is a constant $C$ and a subgroup $H\alnorm G$ such that the following holds: for every $(K,A)$-quasi-isometry $f:G\rightarrow X$  and $g\in G$,  \[d_\Haus(f(gH),D_b)\leq C\] for some $b\in B$.
\end{thm}
To simplify the argument, we only prove Theorem \ref{thm:main fib bundle} in the case where $X$ is coarsely uniformly $n$-connected. This is all that is required if one wishes to prove Theorem \ref{thm:main intro}. The proof is almost identical when $X$ is not necessarily  coarsely uniformly $n$-connected, except that we  need to use the theory of metric complexes rather than metric cell complexes (see Remark \ref{rem:metric complexes}). 

Combining Proposition \ref{prop:fibreqi induces qi between fibres} and Lemma \ref{lem:fibre-preserving qi imp qi between base spaces} with Theorem \ref{thm:main fib bundle} gives the following:
\begin{cor}\label{cor:main cor}
Let $X, F, B, G, H$ be as in Theorem \ref{thm:main fib bundle}. Then $H$ is quasi-isometric to $F$ and $G/H$ is quasi-isometric to $B$.
\end{cor}

We explain how to deduce Theorem \ref{thm:main intro}  from Theorem \ref{thm:main fib bundle}:
\begin{proof}[Proof of Theorem \ref{thm:main intro}]
Suppose that $G$ is a group of type $F_{n+1}$ containing an almost normal coarse $PD_n$ subgroup $H$ with $e(G/H)\geq 3$. Then Theorem \ref{thm:finitenesscoarse} tells us that $G$ is coarsely uniformly $n$-connected and Proposition \ref{prop:almost normal subgroups coarse bundles} ensures that  $p:G\rightarrow G/H$ is a quasi-homogeneous coarse bundle. Thus $p:G\rightarrow G/H$ satisfies the hypotheses of Theorem \ref{thm:main fib bundle}. Using Theorem \ref{thm:main fib bundle} and Corollary \ref{cor:main cor}, we see that if $G'$ is a finitely generated group quasi-isometric to $G$, it contains an almost normal subgroup $H'$ such that $H$ is quasi-isometric to $H'$ and $G/H$ is quasi-isometric to $G'/H'$. In particular, $H'$ is coarse $PD_n$ as required.
\end{proof}

For the remainder of this section, we fix a coarse bundle $p:X\rightarrow B$ as in Theorem \ref{thm:main fib bundle}. The following proposition allows us to apply the coarse topological methods of \cite{margolis2018quasi}.  We fix some $b_0\in B$ and let $D_{b_0}\subseteq X$ be the corresponding fibre, i.e. $p^{-1}(N_E(b_0))$ for $E$ sufficiently large.
\begin{prop}\label{prop:deep imp essential}
Every deep coarse complementary component of $D_{b_0}$ is essential.
\end{prop}
\begin{proof}
Let $C$ be a deep coarse complementary component of $D_{b_0}$.  Lemma \ref{lem:coarse bundle ccc}, ensures $C'\coloneqq N_E(p(C))$ is a deep coarse complementary component of $b_0$ for some $E$. 
By Proposition \ref{prop:quasi geodesic}, we can pick $r$ sufficiently large so that $P^1_r(B)$ is connected and the inclusion $B\rightarrow P^1_r(B)$ is a quasi-isometry.
By the Arzel\`a-Ascoli theorem and quasi-homogeneity, there is a bi-infinite geodesic $l$ in  $P^1_r(B)$ such that $l\cap C'$  contains the vertex set of a ray $r$ of $l$. Since $\phi:l\rightarrow B$ is a coarse embedding, we use Proposition \ref{prop:pullbacks exist} to deduce the existence of a pullback bundle $p_l:X_l\rightarrow l$ with coarse embedding $\hat\phi: X_l\rightarrow X$.

Theorem 11.3 of \cite{kapovich2005coarse} now tells us that   $X_l$ is a coarse $PD_{n+1}$ space. Since being a  coarse $PD_{n+1}$ space is invariant under coarse equivalence (see \cite[Remark 4.4]{margolis2018quasi}),  it follows that $\im(\widehat{\phi})$ is a coarse $PD_{n+1}$ space.  As $l\cap C'$ contains a ray $r$, it follows that $\im(\widehat \phi)\cap C$ contains a coarse $PD_{n+1}$ half-space  whose  boundary has finite Hausdorff distance from  $D_{b_0}$.  Thus Lemma \ref{lem:essential if contains half-space} ensures that  $C$ is essential.
\end{proof}
Since $B$ has at least three ends, we use Lemma \ref{lem:coarse bundle ccc} and Proposition \ref{prop:deep imp essential} to deduce the following:
\begin{cor}\label{cor:3 separating}
The fibre $D_{b_0}$ coarsely separates $X$ into three essential coarse complementary components.
\end{cor}

\begin{defn}\label{defn:unif set}
 A collection $\cW$ of subsets of $X$ is said to be a \emph{uniform $PD_n$-set} if there exist constants $K\geq 1$ and $A\geq 0$ such that the following holds: for every $W\in \mathcal{W}$, there is a $(K,A)$-quasi-isometry  $f_W:X\rightarrow X$ with $d_\Haus( f_W(W),D_{b_0})\leq A$.
\end{defn}

The set $\cW\coloneqq \{D_b\mid b\in B\}$ is a uniform $PD_n$-set because the coarse fibre bundle $p:X\rightarrow B$ is quasi-homogeneous. Our motivation for the definition of uniform $PD_n$-sets is the following lemma:
\begin{lem}\label{lem:unif set close to fibre}
Let $\cW$ be a uniform $PD_n$-set. Then there is a constant $R$ such that for every $W\in \cW$, there is a $b\in B$ such that $d_\Haus(D_b,W)\leq R$. 
\end{lem}

We can use this lemma to prove Theorem \ref{thm:main fib bundle}:
\begin{proof}[Proof of Theorem \ref{thm:main fib bundle}]
Let $X$ and $G$ be as in Theorem \ref{thm:main fib bundle}. We fix a quasi-isometry $f_0:X\rightarrow G$. Then Corollary \ref{cor:3 separating} ensures that $f_0(D_{b_0})$ is coarse $PD_n$ subset of $G$ that coarsely separates $G$ into at least three essential coarse complementary components. Theorem \ref{thm:essential 3-sep} ensures that $f_0(D_{b_0})$ has finite Hausdorff distance from a subgroup $H\leq G$. We fix $K\geq 1$ and $A\geq 0$ and let $\cF$ be the set of \emph{all} $(K,A)$-quasi-isometries from $G$ to $X$. Let $\cW\coloneqq \{ f(H)\mid f\in \cF\}$, which is clearly a uniform $PD_n$-set. Since left multiplication in $G$ is isometric, we see that $\cW=\{ f(gH)\mid f\in \cF, g\in G\}$. It follows from Lemma \ref{lem:unif set close to fibre} that there is constant $R$ such that for every $f\in \cF$ and $g\in G$, there is a $b\in B$ with $d_\Haus(D_b,f(gH))\leq R$ as required. Finally, since $d_\Haus(D_b,D_{b'})<\infty$ for any $b,b'\in B$, we see that $d_\Haus(H,gH)<\infty$ for every $g\in G$, and so Proposition \ref{prop:coset characterisation} ensures $H\alnorm G$.
\end{proof}

We now prove Lemma \ref{lem:unif set close to fibre}. As $X$ is coarsely uniformly $n$-connected, we fix some  uniformly $n$-connected metric cell complex $(X,\bX)$. 
\begin{lem} \label{lem:uniform set parameters}
Let $\mathcal{W}$ be a uniform $PD_n$-set.  Then there exist constants $D$ and $E$ such that:
\begin{enumerate}
\item \label{item:3 components} For every $W\in \mathcal{W}$, there is a decomposition of $X\setminus N_E(W)=C^W_1\cup C^W_2\cup C^W_3$ into disjoint coarse complementary components of $W$ with $\partial C^W_i\subseteq N_E(W)$ for $i=1,2,3$. 
\item \label{item:cohomsupportunif} There are distinct cohomology classes $[\alpha_1^W],[\alpha_2^W]\in \ker (H^{n+1}_c(\bX)\rightarrow H^{n+1}(\bX))$ such that for $i=1,2$ and every $x\in W$, $[\alpha_i^W]$ can be represented by a cocycle supported in $N_D(x)$.
\item For $i=1,2$, the cohomology class $[\alpha_i^W]$ can be represented by a cocycle supported in $C^W_i\backslash N_{E}(W)$.
\item\label{item:crossing} A cohomology class $[\omega]\in \ker (H^{n+1}_c(\bX)\rightarrow H^{n+1}(\bX))$ is represented by cocycles supported in any two of  $C^W_1\backslash N_{E}(W)$, $C^W_2 \backslash N_{E}(W)$ and $C^W_3 \backslash N_{E}(W)$ only if $[\omega]=[0]$. 
\end{enumerate}
\end{lem}
\begin{proof}
We let $\{f_W\mid W\in \cW\}$ be quasi-isometries as in Definition \ref{defn:unif set}. We  recall from Corollary \ref{cor:3 separating} that $D_{b_0}$ coarsely separates $X$ into at least three essential components. Thus there exist essential coarse disjoint, coarse complementary components $C_1$, $C_2$ and $C_3$ of $D_{b_0}$, where we may assume without loss of generality that $C_3=X\setminus (C_1\cup C_2)$. By taking $E$ sufficiently large, we can easily deduce (\ref{item:3 components}) by taking  $C_i^W\coloneqq f_W^{-1}(C_i)\setminus N_E(W)$ for each $i\in \{1,2,3\}$ and $W\in \cW$. 

Parts (\ref{item:cohomsupportunif})--(\ref{item:crossing}) follow from Lemma \ref{lem:cohom class and non-crossing} and the fact that each  quasi-isometry $f_W$  induces a uniformly proper chain homotopy equivalence on $C_\bullet(\bX^{(n+1)})$, whose constants are independent of $W\in \cW$.  For instance, see Theorem 2.7 of \cite{mosher2011quasiactions}, Proposition 9.48 of \cite{drutu2018geometric} or Lemma 3.21 of \cite{margolis2018quasi}.
\end{proof}
Our first step in proving Lemma \ref{lem:unif set close to fibre} is the following:
\begin{lem}\label{lem:coarsely characteristic}
For any quasi-isometry $f:X\rightarrow X$, $d_\Haus(f(D_{b_0}),D_{b_0})<\infty$.
\end{lem}
\begin{proof}
 We observe that $\cW\coloneqq \cup\{D_b\mid b\in B\}\bigcup \cup\{f(D_b)\mid b\in B\}$ is a uniform $PD_n$-set, so we may apply Lemma \ref{lem:uniform set parameters}.  We thus  choose  $D$, $E$, $C_i^W$ and $\alpha_i^W$ so the conclusions of Lemma \ref{lem:uniform set parameters} hold. We suppose that $f$ is a $(K,A)$-quasi-isometry and set $C_i\coloneqq C_i^{D_{b_0}}$ for $i=1$, $2$ and $3$. Since each $C_i$ is deep, we choose $x_i\in C_i$ such that $d(x_i,D_{b_0})>D+E+A$, and a $y_i\in X$ such that $d(f(y_i),x_i)\leq A$. Let $b_i=p(y_i)$ for $i=1,2$.  

Since $f(D_{b_1}),f(D_{b_2})\in \cW$,  Lemma \ref{lem:uniform set parameters} ensures that there exist non-zero cohomology classes $[\alpha_1],[\alpha_2]\in \ker (H^{n+1}_c(\bX)\rightarrow H^{n+1}(\bX))$ such that for every $z_i\in  f(D_{b_i})$, $[\alpha_i]$ can be represented by a cocycle supported in $N_D(z_i)$. As $f(y_i)\in C_i\setminus N_{D+E}(D_{b_0})$, we see that each $[\alpha_i]$ can be represented by a cocycle supported in $C_i\setminus N_{E}(D_{b_0})$ for $i=1,2$. Thus \[f(D_{b_i})\subseteq C_i\cup N_{D+E}(D_{b_0}),\] otherwise $[\alpha_i]$ could be represented by  cocycles supported in  at least two of $C_1\backslash N_{E}(D_{b_0})$, $C_2 \backslash N_{E}(D_{b_0})$ and $C_3 \backslash N_{E}(D_{b_0})$, which would contradict  Lemma \ref{lem:uniform set parameters}.

Since $f(D_{b_1})$ and $f(D_{b_2})$ are at finite Hausdorff distance, we pick $r$ sufficiently large so that $f(D_{b_2})\subseteq N_r(f(D_{b_1}))$. Since $N_r(f(D_{b_1}))\subseteq C_1\cup N_{r+D+E}(D_{b_0})$, we see that 
\begin{align*}
f(D_{b_2})&=f(D_{b_2})\cap N_r(f(D_{b_1}))
\\&\subseteq (C_2\cup N_{D+E}(D_{b_0}))\cap (C_1\cup N_{r+D+E}(D_{b_0}))\subseteq N_{r+D+E}(D_{b_0}).
\end{align*}
Since $d_\Haus(f(D_{b_2}),f(D_{b_0}))<\infty$,  it follows from Theorem \ref{thm:packing} that $f(D_{b_0})$ and $D_{b_0}$ are at finite Hausdorff distance.
\end{proof}

\begin{lem}\label{lem:Wfinitecrit}
Let  $\mathcal{W}$, $D$ and $E$  be as in Lemma  \ref{lem:uniform set parameters}. Suppose that
\begin{enumerate}
\item $D+E<d_\mathrm{Haus}(W,W')<\infty$ for all distinct $W,W'\in \mathcal{W}$;
\item there exist $C\geq 0$ and $x_0\in X$ such that $d(x_0,W)\leq C$ for all $W\in \mathcal{W}$.
\end{enumerate}
Then $\mathcal{W}$ is finite.
\end{lem}
\begin{proof}
Let $\{[\alpha_i^W]\}_{W\in \mathcal{W},i=1,2}$ be the cohomology classes that satisfy the properties of Lemma \ref{lem:uniform set parameters}. Each cohomology class $[\alpha_i^W]$ can be represented by a cocycle supported in $N_{C+D}(x_0)$, thus there are only finitely many such classes\footnote{Recall cohomology is taken with $\bZ_2$ coefficients.}. To show $\mathcal{W}$ is finite, it is sufficient to show that if $W$ and $W'$ are distinct elements of $\mathcal{W}$, then $\{[\alpha_1^W],[\alpha_2^W]\}\neq \{[\alpha_1^{W'}],[\alpha_2^{W'}]\}$. Indeed, if $W\neq W'$, then by reversing the roles of $W$ and $W'$ if necessary we can  choose $x\in W$ such that $d(x,W')>D+E$. Hence both $[\alpha_1^W]$ and $[\alpha_2^W]$ can be represented by cocycles supported in $C_i^{W'}\backslash N_E(W')$ for some $i\in \{1,2,3\}$. This ensures that $\{[\alpha_1^W],[\alpha_2^W]\}\neq \{[\alpha_1^{W'}],[\alpha_2^{W'}]\}$. 
\end{proof}

\begin{proof}[Proof of Lemma \ref{lem:unif set close to fibre}]
The following proof is based on the argument used in the proof of \cite[Proposition 4.3]{vavrichek2013commensurizer}, although in a more general context.
We fix $x_0$, $K$, $A$, $\{f_x\}_{x\in X}$ and $\{\overline{f_x}\}_{x\in X}$ as in Lemma \ref{lem:coarse inverse fibre-preserving}. We set $D$ and $E$ as in Lemma \ref{lem:uniform set parameters}.

 We assume for contradiction that no such $R$ exists, and hence $\cW$ is necessarily infinite. It follows from Lemma \ref{lem:coarsely characteristic} that $d_\Haus(D_{b_0},W)<\infty$ for every $W\in \cW$.  We can thus take a sequence $(W_i)$ in $\mathcal{W}$ such that $\min_{b\in B}d_\mathrm{Haus}(D_{b},W_i)$ is unbounded. Choose $b_i\in B$ such that \[\min_{b\in B}d_\mathrm{Haus}(D_b,W_i)=d_\mathrm{Haus}(D_{b_i},W_i).\]  By passing to a subsequence, we may suppose that \begin{align}d_\mathrm{Haus}(D_{b_j},W_j)> K^2d_\mathrm{Haus}(W_i,D_{b_i})+K(2A+D+E)+3A\label{eqn:close to coset1}\end{align} for all $j> i$. 

We claim that $d_{\mathrm{Haus}}(f_x(W_i), f_{x'}( W_j))>D+E$ for all $i\neq j$ and $x,x'\in X$. We assume for contradiction that $d_{\mathrm{Haus}}(f_x(W_i), f_{x'}( W_j))\leq D+E$ for some $i\neq j$ and $x,x'\in X$. Without loss of generality, we may suppose $j>i$.
	
Since $f_x$ is $A$-fibre preserving, there is a $b\in B$ with $d_\Haus(f_{x}(D_{b_i}),D_{b})\leq A$. Thus
\begin{align*}
d_\mathrm{Haus}(f_{x'}(W_j),D_b)&\leq 
d_\mathrm{Haus}(f_{x'}(W_j),f_{x}(D_{b_i}))+A\\
&\leq d_\mathrm{Haus}(f_{x'}(W_j),f_{x}(W_i))+d_\mathrm{Haus}(f_{x}(W_i),f_{x}(D_{b_i}))+A\\
&\leq K d_\mathrm{Haus}(W_i,D_{b_i})+2A+D+E.\end{align*}
Since $\overline{f_{x'}}$ is also $A$-preserving, there is a $b'\in B$ with $d_\Haus(\overline f_{x'}(D_{b}),D_{b'})\leq A$. Thus as $\overline f_{x'}\circ f_{x'}$ is $A$-close to the identity, we see 
\begin{align*}
d_\mathrm{Haus}(W_j,D_{b'})&\leq 
d_\mathrm{Haus}(\overline {f_{x'}}(f_{x'}(W_j)),\overline {f_{x'}}(D_b))+2A\\
&\leq K d_\mathrm{Haus}(f_{x'}(W_j),D_{b})+3A\\
&\leq K^2d_\mathrm{Haus}(W_i,D_{b_i})+K(2A+D+E)+3A   \end{align*}
It follows from the definition of $b_j$ that \[d_\Haus(W_j,D_{b_j})\leq d_\Haus(W_j,D_{b'})\leq K^2d_\Haus(W_i,D_{b_i})+K(2A+D+E)+3A,\] contradicting (\ref{eqn:close to coset1}) and so proving the claim. 

For each $W_i$, we pick an arbitrary $x_i\in W_i$ and set $\widehat{W}_i\coloneqq f_{x_i}(W_i)$ so that $d(x_0,\widehat W_i)\leq A$. It follows from the above claim that $d_\Haus(\widehat W_i,\widehat W_j)>D+E$ for all $i\neq j$. Thus $\{\widehat W_i\mid i\in \bN\}$ is an infinite uniform $PD_n$-set that contradicts  Lemma \ref{lem:Wfinitecrit}.
\end{proof}

We explain how to reformulate Theorem \ref{thm:main intro} in terms of graphs of groups:
\gogs
\begin{proof}
We claim  a finitely presented group has an almost normal coarse $PD_n$ subgroup $H$ with $e(G/H)\geq 3$ if and only if $G$ splits as a graph of groups satisfying (\ref{item:condition bushy})--(\ref{item:condition vertex group doesn't split}). Corollary \ref{cor:qirigidity of gogs} now follows from Theorem \ref{thm:main intro} and the above claim. If $G$ splits as a graph of groups satisfying  (\ref{item:condition bushy})--(\ref{item:condition vertex group doesn't split}), let $H$ be any edge group. Since all conjugates of $H$ are commensurable to $H$, we deduce that $H\alnorm G$. Corollaries \ref{cor:ends} and \ref{cor:ends tree vs space} then ensure that $e(G/H)\geq 3$.

Conversely, suppose that $G$ contains an almost normal, coarse $PD_n$ subgroup $H$ with $e(G/H)\geq 3$. Theorem \ref{thm:gog} ensures that $G$ splits as a graph of groups satisfying (\ref{item:conjugates are commensurable}) and (\ref{item:condition vertex group doesn't split}). It is clear from the proof of  Theorem \ref{thm:gog} that as $e(G/H)\geq 3$, the corresponding Bass-Serre tree has at least three ends.
\end{proof}

We now prove a partial quasi-isometric classification for such groups. We closely follow the proof of \cite[Theorem 3.1]{papasoglu2002quasi}.
\qiclassification
\begin{proof}
We recall from the proof of Theorem \ref{thm:gog} that there is an infinite ended, $1$-acyclic 2-complex $X$ on which $G$ acts cocompactly. There exists a collection of finitely many $G$-orbits of finite disjoint tracks $\{\tau_i\}_{i\in I}$ such that each  component of $X\backslash \cup_{i\in I}\{ \tau_i\}$ is either compact or 1-ended. The dual tree $T$ gives the Bass-Serre tree of the graph of groups in Theorem \ref{thm:gog}. Similarly, let $X'$ be the corresponding complex for $G'$, let $\{\tau_j\}_{j\in J}$ be the corresponding collection of tracks, and let $T'$ be the dual tree.

Suppose $f:G\rightarrow G'$ is a quasi-isometry. By Theorem \ref{thm:main fib bundle} $f$ is fibre-preserving, so  Lemma \ref{lem:fibre-preserving qi imp qi between base spaces} ensures $f$ induces a quasi-isometry $\hat f: X\rightarrow X'$.
It follows from  the proof of \cite[Theorem 3.1]{papasoglu2002quasi} that $\hat f$ sends 1-ended components of  $X\backslash \cup\{ \tau_i\}$ to a uniform metric neighbourhood of a 1-ended component of $X'\backslash \cup \{\tau_j\}$.  It was shown in the proof of Theorem \ref{thm:gog} that these components are quasi-isometric to quotient spaces of the form $G_v/H_v$ or $G'_{v'}/H'_{v'}$. Moreover, as  $d_\mathrm{Haus}(\hat f(gH),f(gH))\leq B$ it follows that $f$ sends each   $G_v\in V_\infty \mathcal{G}$ to within uniform Hausdorff distance of some $G'_{v'}\in V_\infty \mathcal{G}'$ and vice-versa. Proposition \ref{prop:fibreqi induces qi between fibres} then ensures $f$ induces the required quasi-isometries.
\end{proof}

%% file: central.tex
\section{Fibre distortion}\label{sec:applications}

Theorem \ref{thm:main intro}, Corollary \ref{cor:qirigidity of gogs} and Theorem \ref{thm:qiclassification of gogs} give necessary conditions for two groups in a certain class to be quasi-isometric. Can we obtain finer QI invariants that yield finer QI classification and rigidity results?
  For example,  observe that $BS(1,3)=\langle a,t\mid tat^{-1}=a^3\rangle $ and  $F_2 \times \mathbb{Z}$ have  the same Bass-Serre tree in the following sense: both act on 4-valent trees with all edge and vertex stabilisers isomorphic to $\mathbb{Z}$. Thus Corollary \ref{cor:qirigidity of gogs} and Theorem \ref{thm:qiclassification of gogs} cannot  distinguish them up to quasi-isometry. 
  
To distinguish such groups, we use the notion of height change defined in \cite{farbmosher1999bs2} and \cite{whyte2001baumslag}. Let $H\coloneqq \langle a \rangle \alnorm BS(1,3)$. We think of $ BS(1,3)$ as fibred over its left $H$-coset space with  fibres corresponding to $H$-cosets. Endowing each fibre with the path metric, the closest point projection from $H$ to $t^k H$ coarsely  distorts distances by a factor of $3k$; this is because $t^ka^i=a^{3i}t^k$. In contrast, no such distortion can occur in $F_2\times \mathbb{Z}$. These sort of phenomena will be used to prove Theorems \ref{thm:zfibre} and \ref{thm:zn fibre with aiq}.

We fix a finitely generated group $G$ containing a finitely generated, almost normal subgroup $H\alnorm G$. 
 For each $g\in G$, let $p_g:H\rightarrow gH$ be a \emph{closest point projection map}, i.e. any map such that $d(x,p_g(x))\leq d(x,y)$ for all $y\in gH$. We remark that by Proposition \ref{prop:coset characterisation}, $H$ and $gH$ are at finite Hausdorff distance. Any map $f:H\rightarrow gH$ which moves points a uniform distance will be close to $p_g$, i.e. if $\sup_{h\in H}d(f(h),h)<\infty$, then $\sup_{h\in H}d(f(h),p_g(h))<\infty$. Thus $p_g$ is coarsely well-defined.

\begin{exmp}\label{exmp:abelian p_g}
Suppose $\mathbb{Z}[v_1,\dots, v_n]=\mathbb{Z}^n\cong H\alnorm G$. Since $gHg^{-1}$ and $H$ are commensurable, for each $v_j$ there exist integers $q_j,p_{1j}, \dots, p_{2j}$ such that  $v_j^{q_j} =gv_1^{p_{1j}}v_2^{p_{2j}}\dots v_n^{p_{nj}}g^{-1}$. We associate to each $g\in G$ a matrix $A_g\in GL(n,\mathbb{Q})$, whose $(i,j)$ entry is $\frac{p_{ij}}{q_j}$. Then  \[\sup_{h\in H} d(p_g(h),g(\lfloor A_g\cdot h\rfloor))<\infty,\] where $\lfloor (x_1,\dots, x_n)\rfloor\coloneqq (\lfloor x_1\rfloor,\dots, \lfloor x_n\rfloor)$. We show this in the case $n=1$; the general case is similar but notationally unwieldy.

We write $v_1=v$, $q_1=q$ and $p_{11}=p$. For each $k\in \mathbb{Z}$, let $r_k\coloneqq k-q\lfloor \frac{k}{q} \rfloor$. Then $v^k=v^{q\lfloor \frac{k}{q} \rfloor+r_k}=gv^{p\lfloor \frac{k}{q} \rfloor}g^{-1}v^{r_k}$.  Since \[\sup_{k\in \mathbb{Z}} d(gv^{\lfloor \frac{p}{q} k\rfloor},gv^{p\lfloor \frac{k}{q} \rfloor}g^{-1}v^{r_q})<\infty,\] we see that $\sup_{k\in \mathbb{Z}} d(gv^{\lfloor \frac{p}{q} k\rfloor},v^k)<\infty.$ We deduce from the above discussion that \[\sup_{k\in \mathbb{Z}} d(p_g(v^k),gv^{\lfloor \frac{p}{q} k\rfloor})<\infty.\]
\end{exmp}
 
We pick a finite generating set for $H$ and equip $H$ with the word metric $d_H$. For each $g\in G$, the map $H\rightarrow G$ given by $h\mapsto gh$ is a coarse embedding. This follows from Lemma \ref{lem:sbgpis coarsely embedding} and the fact that left multiplication by $g$ is isometric. Thus one equips each coset $gH$ with the metric $d_{gH}(gh,gh')\coloneqq d_H(h,h')$.
Proposition \ref{prop:fibreqi induces qi between fibres} ensures that $p_g:H\rightarrow gH$ is a quasi-isometry with respect to the above metrics. 
\begin{rem}
If $gH=kH$, then for all $h,h'\in H$ \[d_{gH}(gh,gh')=d_{H}(h,h')=d_H(k^{-1}gh,k^{-1}gh')=d_{kH}(gh,g'h).\] Thus  $d_{gH}$ depends only on the coset $gH$ and is independent of $g$. 
\end{rem}

\begin{defn}
Let $X$ and $Y$ be metric spaces and $f:X\rightarrow Y$ be a quasi-isometry. We define \[\kappa(f)\coloneqq \inf \{K\geq 1 \mid \textrm{$f$ is a $(K,A)$-quasi-isometry for some $A\geq 0$} \}.\] 
Let $G$ be a finitely generated group with $H\alnorm G$ a finitely generated almost normal subgroup.  We define a function $F:G/H\rightarrow \mathbb{R}$ by $gH \mapsto\log(\kappa(p_g))$, which we call the \emph{fibre distortion function}.
\end{defn}

\begin{exmp}\label{exmp:fibre distortion free abelian}
Suppose $H\alnorm G$ is a finitely generated free abelian subgroup. We note that the word metric with respect to a $\mathbb{Z}$-basis of $H$ is simply the $l_1$-norm of $H$. Define $A_g$ as in Example \ref{exmp:abelian p_g}. Choose $B$ such that \[\sup_{h\in H} d(p_g(h),g(\lfloor A_g\cdot h\rfloor))\leq B.\] We also observe that $\lvert \lfloor A_g\cdot h\rfloor -A_g\cdot h\rvert_1\leq n$ for all $h\in \bR^n$.
 Thus for any $h\neq h'\in H$, \[\Bigg\lvert \frac{d_{gH}(p_g(h),p_g(h'))}{d_H(h,h')}-\frac{\lvert A_g(h'-h)\rvert_1}{\lvert h'-h\rvert_1}\Bigg\rvert\leq \frac{2B+2n}{d_H(h,h')},\] and so 
\begin{align*}
 K(h,h')d_H(h,h')-A&\leq d_{gH}(p_g(h),p_g(h')) \leq K(h,h')d_H(h,h')+A.
\end{align*}
where $K(h,h')\coloneqq\frac{\lvert A_g(h'-h)\rvert_1}{\lvert h'-h\rvert_1}$ and $A\coloneqq 2B+2n$. Notice that \[\frac{1}{\lVert A_g\rVert_1}\leq K(h,h')\leq \lVert A_g\rVert_1,\] where  $\lVert A_g\rVert_1$ is the matrix norm of $A_g$. Thus $F(gH)=\lvert\log( \lVert A_g\rVert_1)\rvert$.
\end{exmp}

We now show that fibre distortion is a  quasi-isometry invariant up to some additive error. 
\begin{prop}\label{prop:F qi invariant}
Suppose that $G$ and $G'$ are groups containing  almost normal subgroups $ H\alnorm G$ and $ H'\alnorm G'$ such that the hypotheses of Theorem \ref{thm:main intro} are satisfied. Let $F$ and $F'$ be the associated fibre distortion functions. If  $f:G\rightarrow G'$ is a quasi-isometry and $\widehat f: G/H\rightarrow G'/H'$ is the induced quasi-isometry as in Lemma \ref{lem:fibre-preserving qi imp qi between base spaces}, there is a constant $C\geq 0$ such that 
\[\lvert F'(\widehat{f}(gH) )-F(gH)\rvert\leq  C.\]
\end{prop}
The existence of $\widehat{f}$ follows from Theorem \ref{thm:main fib bundle}, since any quasi-isometry $f:G\rightarrow G'$ is fibre-preserving. To prove Proposition \ref{prop:F qi invariant} we make use of the following lemma:
\begin{lem}\label{lem:qi bounds composition}
 Suppose that $f:X\rightarrow Y$ is a quasi-isometry. If $g:W\rightarrow X$ and $h:Y\rightarrow Z$ are $(K,A)$-quasi-isometries, then \[\frac{\kappa (f)}{K^2}\leq \kappa(h\circ  f\circ g))\leq \kappa (f) K^2.\]
\end{lem}
\begin{proof}
Let $\epsilon>0$ and $M\coloneqq \kappa (f)$. Then there exists a $B\geq 0$ such that $f$ is a $(M+\epsilon, B)$-quasi-isometry. Thus $h\circ f\circ g$ is a $(K^2(M+\epsilon),B')$-quasi-isometry for some $B'> 0$. This shows that $\kappa(h\circ f\circ g)\leq K^2M=K^2 \kappa(f)$.

To obtain the lower bound, we observe that as $f$ is not an $(M-\epsilon, C)$-quasi-isometry for any $C$, there exist sequences of points  $(x_n),(x'_n)$ such that for each $n$, either \[d(f(x_n),f(x'_n))> (M-\epsilon) d(x_n,x'_n)+n\] or \[d(f(x_n),f(x'_n))< \frac{1}{M-\epsilon} d(x_n,x'_n)-n.\] There exist sequences $(w_n)$ and $(w'_n)$ in $W$ such that for all $n$, $d(g(w_n),x_n)\leq A$ and $d(g(w'_n),x'_n)\leq A$. Thus there is a constant $B''> 0$ such that for every $n$, either
\[d((h\circ f\circ g)(w_n),(h\circ f\circ g)(w'_n))> \frac{M-\epsilon}{K^2} d(w_n,w'_n)+\frac{n}{K}+B''\] or \[d((h\circ f\circ g)(w_n),(h\circ f\circ g)(w'_n))< \frac{K^2}{M-\epsilon} d(w_n,w'_n)-nK-B''.\] 
By taking $n$ sufficiently large, we deduce that $h\circ f\circ g$ is not an  $(\frac{M-\epsilon}{K^2},A')$-quasi-isometry for any $A'\geq 0$.  Thus $\kappa(h\circ f\circ g)\geq \frac{M}{K^2}=\frac{\kappa(f)}{K^2}$.
\end{proof}

\begin{proof}[Proof of Proposition \ref{prop:F qi invariant}]
Let $f:G\rightarrow G'$ be a $(K,A)$ quasi-isometry. By Theorem \ref{thm:main fib bundle} and Lemma \ref{lem:fibre-preserving qi imp qi between base spaces}, there is a  constant $C\geq 0$ such that:
\begin{enumerate}
\item $d_\mathrm{Haus}(f(H),H')\leq C$;
\item for all $g\in G$, then $d_\mathrm{Haus}(f(gH),\widehat{f}(gH))\leq C$.
\end{enumerate}

 We thus define $f_0:H\rightarrow H'$ such that $d(f_0(h),f(h))\leq C$ for all $h\in H$.   We fix $g\in G$, and similarly define $f_g:gH\rightarrow \widehat f(gH)$ such that $d(f_g(gh),f(gh))\leq C$ for all $h\in H$.
Let $p_g:H\rightarrow gH$ and $p'_{g}:H'\rightarrow \widehat{f}(gH)$ be closest point projections. By Proposition \ref{prop:fibreqi induces qi between fibres} there exist constants $K'\geq 1$ and $A'\geq 0$, independent of $g$, such that $f_0$ and $f_g$ are $(K',A')$ quasi-isometries for all $g\in G$.  Moreover, we may choose $(K',A')$ sufficiently large so that $f_0$ has a quasi-inverse $\bar f_0$ that  is also a $(K',A')$-quasi-isometry. 

Since $\sup_{h\in H}d(f_g\circ p_g(h),f(h))<\infty$ and $\sup_{h\in H}d(p'_g\circ f_0(h),f(h))<\infty$, we see that  $\sup_{h\in H} d(p'_g\circ f_0(h),f_g\circ p_g(h))<\infty$. Thus \[\sup_{h'\in H'} d( p'_g(h'), (f_g\circ p_g\circ \bar f_0)(h'))<\infty.\] It follows from Lemma \ref{lem:qi bounds composition} that $\frac{\kappa(p_g)}{(K')^2}\leq\kappa(p'_g)=\kappa(f_g\circ p_g\circ \bar f_0)\leq (K')^2\kappa(p_g)$ and so  $\lvert F'(\widehat f(gH))-F(gH)\rvert\leq 2\log (K')$.
\end{proof}

In the case where $H$ is infinite cyclic, we define a related invariant called the height function, generalising the height function used in \cite{whyte2001baumslag}.
\begin{defn}
Let $G$ be a finitely generated group containing an almost normal subgroup $\mathbb{Z}\cong \langle t \rangle =H\alnorm G$. 
Pick $q,p\in \mathbb{Z}$ such that $t^q=gt^pg^{-1}$. The \emph{height function} $h:G\rightarrow \mathbb{R}$ is defined by $g \mapsto\log(\lvert \frac{p}{q} \rvert)$.
\end{defn}
The height function is well-defined as the ratio $\frac{p}{q}$ depends only on $g$. The following is evident from Examples \ref{exmp:abelian p_g} and \ref{exmp:fibre distortion free abelian}: 
\begin{lem}\label{lem:height vs F}
If $H\alnorm G$ is infinite cyclic, then $\lvert h(g) \rvert =F(gH)$.
\end{lem}

One advantage of the height function over the fibre distortion function is the following:
\begin{prop}\label{prop:height is hom}
The height function $h:G\rightarrow (\mathbb{R},+)$ is  a homomorphism. 
\end{prop}
\begin{proof}
Let $g,k\in G$. We pick non-zero integers $n,m,r,s$ such that $t^m=gt^ng^{-1}$ and $t^r=kt^sk^{-1}$. Thus $t^{mr}=gt^{nr}g^{-1}=gkt^{ns}k^{-1}g^{-1}$ and so \[h(gk)=\log\Big(\Big\lvert\frac{ns}{mr}\Big\rvert\Big)=\log\Big(\Big\lvert\frac{n}{m}\Big\rvert\Big)+\log\Big(\Big\lvert\frac{s}{r}\Big\rvert\Big)=h(g)+h(k).\] As $g^{-1}t^mg=t^n$, we see that $h(g^{-1})=\log(\lvert \frac{m}{n}\rvert)=-\log(\lvert \frac{n}{m}\rvert)=-h(g)$. Thus $h$ is a homomorphism.
\end{proof}

We now show quasi-isometric rigidity for finitely presented $\mathbb{Z}$-by-($\infty$-ended) groups. This result also holds for groups of type $FP_2$.
\zfibre
\begin{proof}
Suppose $G$ is a finitely presented group that contains an infinite cyclic normal subgroup $H$, with $G/H$ infinite ended.   As $H\vartriangleleft G$, the height function $h:G\rightarrow \mathbb{R}$ is zero. Thus the fibre distortion function $F$ is zero by Lemma \ref{lem:height vs F}. Suppose $G'$ is a finitely generated group quasi-isometric to $G$. By Theorem \ref{thm:main intro}, $G'$ contains an 2-ended almost normal subgroup $H'\alnorm G'$. Since 2-ended groups are virtually cyclic, we may assume by Proposition \ref{prop:alnorm is comm inv} that $H'=\langle t \rangle$ for some $t\in G'$. By Proposition \ref{prop:F qi invariant} and Lemma \ref{lem:height vs F}, it follows that  the associated height function $h':G'\rightarrow \mathbb{R}$ has bounded image. Since $(\bR,+)$ is torsion-free,  Proposition \ref{prop:height is hom} ensures $h'$ is  trivial.

Let $S=\{s_1,s_2,\dots, s_n\}$ be a finite generating set of $G'$. Since $s_iHs_i^{-1}$ and $H$ are commensurable,  there are positive integers $n_i$ and $m_i$ such that $s_it^{m_i}s_i^{-1}=t^{\pm n_i}$. As $h'$ is trivial, $n_i=\pm m_i$. Thus for sufficiently large $M$ such that $m_i\vert M$ for all $i$, we deduce that $s_i\langle t^M\rangle s_i^{-1}=\langle t^M\rangle$ for all $i$. Since $S$ is a generating set,  $\langle t^M\rangle$ is an infinite cyclic normal subgroup of $G'$.
\end{proof}
It is easy to see, as observed in the preceding proof, that central extensions necessarily have bounded fibre distortion. 
One may wonder if the converse holds: does bounded fibre distortion distortion necessarily imply the group is (virtually) a central extension? This is not the case. The following example resolves a question of \cite[Section 12.2]{frigerio2015graphmflds}, giving an example of a group quasi-isometric to $F_2\times \mathbb{Z}^2$ that does not have a normal free abelian subgroup of rank 2. 
\begin{exmp}\label{exmp:qitofreeXz2}
Let \[G=\langle a,b,t\mid [a,b], a^{13}=ta^{5}b^{12}t^{-1},b^{13}=ta^{-12}b^5t^{-1}\rangle.\] This is an HNN extension of $\mathbb{Z}^2$ with stable letter $t$. Notice that the associated matrix 
\[A\coloneqq\begin{pmatrix}
\frac{5}{13} & -\frac{12}{13}\\[6pt]
\frac{12}{13}& \frac{5}{13}
\end{pmatrix}\] is an infinite order rotation through angle $\arccos(\frac{5}{13})$. Thus the $\mathbb{Z}^2$ subgroup $H=\langle a,b\rangle\leq G$ is almost normal but not normal.   We now build a graph of spaces $X$ by taking two tori $T$ and $T'$ and gluing $T'\times [0,1]$ to $T$ so that $\pi(X)=G$. Endowing $T$ and $T'\times [0,1]$ with standard Euclidean metrics, we see that as $A$ is orthogonal, the attaching maps used to define $X$ are isometric embeddings. Thus $\tilde X$ is isometric to $\mathbb{E}^2\times T$, where $T$ is a regular tree with infinitely many ends, and so $G$ is quasi-isometric to $F_2\times \mathbb{Z}^2$. Such examples are considered in work of Leary and Minasyan \cite{learyminasyan19}.  See also \cite{huang2018commensurators}.
\end{exmp}

We now investigate what one can say if the fibre distortion function is unbounded. In particular, we are interested in the following situation:
\begin{defn}
Let $G$ be a finitely generated group containing an almost normal subgroup $H\alnorm G$. We say that $(G,H)$ has \emph{proper fibre distortion} if the fibre distortion map $F:G/H\rightarrow \mathbb{R}$ is proper, i.e. for each $r\geq 0$, $F^{-1}([1,r])$ is finite. 
\end{defn}
\begin{exmp}
Suppose $G$ is a finitely generated group extension \[1\rightarrow H\cong \mathbb{Z}^n\rightarrow G\rightarrow Q\rightarrow 1\] such that the map $Q\rightarrow Aut(H)\cong GL(n,\mathbb{Z})$ has finite kernel. In such a situation, we say that $G$ is \emph{$\mathbb{Z}^n$--by--($\infty$ ended) with almost injective quotient (AIQ)}. Observe that for every $r\geq 0$, there are only finitely many matrices $M\in GL(n,\mathbb{Z})$ such that $\lVert M\rVert_1\leq r$.  Thus Example \ref{exmp:fibre distortion free abelian} ensures that $(G,H)$ has proper fibre distortion if and only if $G$ is $\mathbb{Z}^n$--by--($\infty$ ended) with AIQ.
\end{exmp}

Proposition \ref{prop:F qi invariant} implies the following:
\begin{cor}\label{cor:proper fibre distortion}
Suppose that $G$ and $G'$ are groups containing  almost normal subgroups $ H\alnorm G$ and $ H'\alnorm G'$ such that the hypotheses of Theorem \ref{thm:main intro} are satisfied. If $G$ and $G'$ are quasi-isometric, then $(G,H)$ has proper fibre distortion if and only if $(G',H')$ does.
\end{cor}

The following is a  generalisation of Lemma 4.1 of \cite{whyte2010coarse}:
\begin{prop}\label{prop:proper fib distortion}
Suppose $G$ is finitely presented, $\mathbb{Z}^n\cong H\alnorm G$  and $e(G/H)=\infty$. Suppose $(G,H)$ has proper fibre distortion. Then $H$ has a finite index subgroup that is normal in $G$.
\end{prop}
\begin{proof}
Let $F:G/H\rightarrow \mathbb{R}$ be the fibre distortion map. It follows from the definition of $A_g$ in Example \ref{exmp:abelian p_g} that $A_{gk}=A_kA_g$ for all $g,k\in G$.  Suppose $g\in G$ and $h\in H$.  Using Example \ref{exmp:abelian p_g}, we see that 
\[\mathrm{sup}_{k\in H}d(k,ghg^{-1}\lfloor A_{g^{-1}}A_{h}A_g\cdot k\rfloor)=\mathrm{sup}_{k\in H}d(k,ghg^{-1}\lfloor A_{ghg^{-1}}\cdot k\rfloor)<\infty.\] Since $h\in H$, $A_h$ is the identity, so $A_{g^{-1}}A_{h}A_g$ is also identity. By Example \ref{exmp:fibre distortion free abelian}, we see that $F(ghg^{-1}H)=0$.

Let $X$ be the 2-complex  with vertex set $G/H$ constructed in the proof of Theorem \ref{thm:gog}. As in the  proof of Theorem \ref{thm:gog}, there exists a collection of tracks $\{\tau_i\}$ in $X$ such that components of $X\backslash \cup_i \tau_i$ correspond to vertices of the dual tree $T$. 
We claim that $H$ fixes $T$. If this is the case, then let $K$ be the kernel of the action of $G$ on $T$. By assumption $H\leq K$. Since edge stabilisers are commensurable to $H$ and $K$ is contained in an edge stabiliser, it follows that $K$ is commensurable to $H$.  Since $K$ is normal, $H$ has a finite index normal subgroup.

To prove the claim, we suppose some $h\in H$ doesn't fix $T$. Let $v$ be the vertex of $T$ corresponding to the component of $X\backslash \cup_i \tau_i$ containing $H$. Then $Hv=v$. Thus there is an $A\geq 0$ and a component $C$ of $T\backslash N_A(v)$ such that $hC\neq C$. Since $G\backslash T$ is cocompact, we can pick a sequence $(g_i)$ in $G$ such that  $g_iv\in C$ for all $i$ and $\lim_{i\rightarrow \infty} d_T(v,g_iv)=\infty$. As any path from $g_iv$ to $hg_iv$ passes through $N_A(v)$, we see that \[d_T(v,g_i^{-1}hg_iv)=d_T(g_iv,hg_iv)\geq d_T(g_iv,v) + d_T(g_iv,v)-2A.\]  Thus $\lim_{i\rightarrow \infty} d_T(v,g_i^{-1}hg_iv)=\infty$. By passing to a subsequence, we may assume that $g_i^{-1}hg_iv\neq g_j^{-1}hg_jv$ unless $i=j$. Since $Hv=v$, we deduce that $g_i^{-1}hg_iH\neq g_j^{-1}hg_jH$ when $i\neq j$. By the first paragraph, we have $F(g_i^{-1}hg_iH)=0$ for all $i$, which contradicts the properness of $F$.\end{proof}

We now combine Theorem \ref{thm:main intro} with Propositions \ref{prop:F qi invariant} and \ref{prop:proper fib distortion} to deduce quasi-isometric rigidity for the class of $\mathbb{Z}^n$--by--($\infty$ ended) groups with AIQ:
\znfibre

%% file: fibreqirigidity.tex
\section{QI rigidity of surface group extensions}\label{sec:qirigidity surfaces}
Throughout this appendix, let $S$ denote a closed hyperbolic surface.  In \cite{farbmosher2002surfacebyfree}, Farb and Mosher prove quasi-isometric rigidity for hyperbolic surface--by--free groups. These are group extensions of form $1\rightarrow \pi_1(S)\rightarrow \Gamma_H\rightarrow K\rightarrow 1$, where  the induced action $K\leq \Out(\pi_1(S))\cong\MCG(S)$ is a Schottky subgroup, i.e. a convex cocompact free subgroup in the sense of \cite{farbmosher02convexcocompact}. These methods were further developed by Mosher \cite{mosher2003fiber, mosher2009homology} to study the coarse geometry of other surface group extensions. 

Mosher classified \emph{fibre-preserving} quasi-isometries of certain surface group extensions in \cite{mosher2003fiber}. Theorem \ref{thm:main fib bundle} tells us that in many cases, every quasi-isometry is fibre-preserving. Combining these results allows us to deduce Theorem \ref{thm:qirigidity surface group extension}. For the convenience of the reader, we give an account  of Mosher's methods and explain how to deduce Theorem \ref{thm:qirigidity surface group extension} from Theorem \ref{thm:main fib bundle} and \cite{mosher2003fiber}. We refer the reader to \cite{farbmosher02convexcocompact}, \cite{farbmosher2002surfacebyfree} and \cite{mosher2003fiber} for more details on these topics.

\subsection*{Background on Teichm\"uller space}
\emph{Teichm\"uller space} is defined to be the space of marked hyperbolic structures on $S$ modulo isotopy. Equivalently, it is the space of marked conformal structures on $S$ modulo isotopy. 
A \emph{measured foliation} on $S$ is defined to be a foliation $\cF$ with finitely many singular points such that:
\begin{enumerate}
\item each neighbourhood of a singular point is a $k$-pronged saddle for $k\geq 3$ (see \cite{farbmargalitbook});
\item $\cF$ is equipped with a positive transverse Borel measure.
\end{enumerate}
A \emph{saddle collapse} of a measured foliation $\cF$ is obtained by collapsing the leaf segment joining two singularities to a point. The set of measured foliations of $S$, modulo isotopy and saddle collapse, is denoted $\MF$. Given a measured foliation $\cF$ and $\lambda >0$, let $\lambda\cF$ be the measured foliation obtained by multiplying the transverse measure by  $\lambda$. Let $\PMF$ denote the set of all projective measured foliations, i.e.  $\MF$ modulo the equivalence relation $\lambda F\sim F$ for $\lambda>0$. We let $[\cF]\in \PMF$ denote the equivalence class containing $\cF$.

We define a topology on $\PMF$ and $\cT$ as follows. Let $\cC$ denote the set of isotopy classes of simple closed curves on $S$ and let $\bP[0,\infty)^\cC$ be the projectivisation of $[0,\infty)^\cC$. There is an injection $\cT\rightarrow \bP[0,\infty)^\cC$  that associates to each $x\in \cT$ and $c\in \cC$ the length of the unique geodesic in the class $c$. The image of this map is homeomorphic to an open ball of dimension $6g-6$. For each measured foliation $\cF$ and $c\in \cC$, we define $\ell_\cF(c)$ to be the infimum  of $\int_{\gamma}\cF$ as $\gamma$ ranges over curves in the isotopy class $c$. The map $\cF\mapsto (c\mapsto \ell_\cF(c))$ descends to a well-defined map $\PMF\rightarrow \bP[0,\infty)^\cC$ with image homeomorphic to a sphere of dimension $6g-7$. We topologise $\cT$ and $\PMF$ by identifying them with their images in $\bP[0,\infty)^\cC$ endowed with subspace topology. Then Thurston's compactification theorem says that with respect to the above embedding into $\bP[0,\infty)^\cC$, $\overline{\cT}=\cT\sqcup \PMF$ is the closure of $\cT$.

We now describe the Teichm\"uller metric on $\cT$. We say that two measured foliations $\cF_x$ and $\cF_y$ are \emph{transverse} if:
\begin{enumerate}
\item they have the same set of singular point;
\item each singular point has the same number of prongs;
\item their leaves are transverse at all regular points.
\end{enumerate}
Two transverse measured foliations $\cF_x$ and $\cF_y$ determine a conformal structure at regular points of $S$: a neighbourhood of each regular point is identified with a subset of $\mathbb{C}$ such that leaves of $\cF_x$ (resp. $\cF_y$) correspond to lines with constant imaginary (resp. real) part  and distance is determined by the transverse measures. This conformal structure extends uniquely over singular points to a conformal structure on $S$ itself, hence determines a point of $\cT$ which we denote $\sigma(\cF_x,\cF_y)$. This is well-defined for any pair of transverse measured foliations  $(\cF_x,\cF_y)$ that represent the same pair  in $\PMF^2$.

A pair of transverse measured foliations $(\cF_x,\cF_y)$ defines a \emph{Teichm\"uller geodesic} \[t\mapsto \sigma(e^t\cF_x,e^{-t}\cF_y),\] which we denote $\geodesic{[\cF_x]}{[\cF_y]}$. Teichm\"uller's theorem says that every pair of points on Teichm\"uller space lie on a unique Teichm\"uller geodesic, and so the associated \emph{Teichm\"uller metric} given by $d(\sigma(e^t\cF_x,e^{-t}\cF_y),\sigma(e^s\cF_x,e^{-s}\cF_y))=\lvert s-t\rvert$ is  a well-defined metric on $\cT$. The topology induced by this metric agrees with the topology on $\cT$ defined by homeomorphically identifying $\cT$ with a subset of  $\bP[0,\infty)^\cC$ via the preceding embedding. Henceforth, it will be assumed that $\cT$ is endowed with the Teichm\"uller metric.
\subsection*{The mapping class group and its action on Teichm\"uller space}
The \emph{mapping class group} $\mathrm{MCG}(S)$ is the group $\mathrm{Homeo}(S)/\mathrm{Homeo}_0(S)$, where $\mathrm{Homeo}(S)$ is the group of homeomorphisms (not necessarily orientation preserving) of $S$ and $\mathrm{Homeo}_0(S)$ is the connected component of the identity. This is sometimes called the \emph{extended mapping class group} of $S$. The mapping class group $\MCG(S)$ acts properly discontinuously and isometrically on $\cT(S)$, but the action is not cocompact.

Let $p\in S$ be a point. The Dehn--Nielsen--Baer says that there is an isomorphism of short exact sequences, where the top row is known as the \emph{Birman exact sequence} associated to the mapping class group.

\begin{equation}
\begin{tikzcd}
1 \arrow[r]   & \pi_1(S)\arrow[r] \arrow[d,"="] &\MCG(S\setminus \{p\})\arrow[r,"\pi"]\arrow[d,"\cong"]&\mathrm{MCG}(S)\arrow[r]\arrow[d,"\cong"]& 1\\
1\arrow[r]&\pi_1(S)\arrow[r]&\mathrm{Aut}(\pi_1(S))\arrow[r]&\mathrm{Out}(\pi_1(S))\arrow[r]& 1
\end{tikzcd}
\end{equation}

 A homeomorphism $H\in \Homeo(S)$ is said to be a \emph{pseudo-Anosov} if there exist a pair of transverse measured foliations $(\cF_x,\cF_y)$ and a $\lambda>1$ such that $f(\cF_x,\cF_y)=(\lambda\cF_x,\lambda^{-1}\cF_y)$. A mapping class $f\in \MCG(S)$ is said to be  \emph{pseudo-Anosov} if it has a pseudo-Anosov representative. Every pseudo-Anosov $g$ acts by translations of non-zero length along a Teichm\"uller geodesic. This geodesic is unique and is called the \emph{axis} of $g$. 

This  action extends to an action  $\MCG(S) \curvearrowright \overline{\cT(S)}$. Given $f\in \MCG(S)$,  let $\Fix{f}$ denote the set of fixed points of $f$ in $\overline{\cT(S)}$. If $f$ is a pseudo-Anosov such that $f(\cF_x,\cF_y)=(\lambda\cF_x,\lambda^{-1}\cF_y)$ for some $\lambda>1$, then $\Fix{f}=\{[\cF_x],[\cF_y]\}$.
\begin{lem}[\cite{mccarthy1994normalizers}]
Given two pseudo-Anosovs $f,g\in \MCG(S)$, either $\Fix{f}=\Fix{g}$ or $\Fix{f}\cap\Fix{g}=\emptyset$. The setwise stabiliser of $\Fix{f}$ is equal to $\Comm(\langle f \rangle)$ and  contains $\langle f\rangle\cong \mathbb{Z}$ as a finite index subgroup.
\end{lem}

We say that $L\leq \MCG(S)$ is \emph{irreducible} if it doesn't preserve a finite collection of disjoint simple closed curves.  Ivanov showed that a subgroup of $\mathrm{MCG}(S)$ is infinite and irreducible if and only if it contains a pseudo-Anosov  \cite[Corollary 7.14]{ivanov92subgroups}. 

Given a measured foliation $\cF$ of $S$, we can lift it to a $\pi_1(S)$-invariant measured foliation $\widetilde{\cF}$ of $\widetilde S\cong \Hyp$. We set $E([\cF])\subseteq S^1$ to be the set of endpoints of leaves of $\widetilde{\cF}$, noting that $E([\cF])$ is independent of the choice of $\cF\in [\cF]$. If $(\cF_x,\cF_y)$ is a pair of transverse measured foliations, then $E([\cF_x])\cap E([\cF_y])=\emptyset$ by transversality of $\cF_x$ and $\cF_y$. Thus if $[\cF_x],[\cF_y]\in \PMF$ can be joined by a Teichm\"uller geodesic, then $E([\cF_x])\cap E([\cF_y])=\emptyset$.

Let $g$ be a pseudo-Anosov with $\Fix{g}=\{[\cF_x],[\cF_y]\}$. We define $E(g)\coloneqq E([\cF_x])\sqcup E([\cF_y])$. Suppose that $f$ is another pseudo-Anosov such that $\Fix{f}\cap \Fix{g}=\emptyset$. Then any two distinct elements of $\Fix{f}\cup \Fix{g}$ can be joined by a Teichm\"uller geodesic. To see this, we recall from  \cite{farbmosher02convexcocompact} that for $n$ sufficiently large, the subgroup generated by $\langle f^n, g^n\rangle$ is convex cocompact. The definition of convex cocompactness then ensures that  any two elements of  $\Fix{f^n}\cup \Fix{g^n}=\Fix{f}\cup \Fix{g}$ can be joined by a Teichm\"uller geodesic. Hence we deduce the following:

\begin{lem}\label{lem:endpts of leaves disjoint}
If $f,g\in\MCG(S)$ are pseudo-Anosovs, then either    $E(f)\cap E(g)=\emptyset$ or $E(f)=E(g)$. In the latter case, $g\in \Comm(\langle f\rangle)$ and $f$ and $g$ share an axis.
\end{lem}

A subset $A\subseteq \cT$ is said to be \emph{cobounded} if there is some bounded set $\Omega\subseteq \cT$ such that $A\subseteq\MCG(S)\Omega$. Every axis of a pseudo-Anosov is cobounded.  We deduce the following lemma  via Proposition 3.3 of \cite{anderson2007freesubgpmcgs}, which also follows from the proof of \cite[Lemma 2.4]{farbmosher02convexcocompact}.
\begin{lem}\label{lem:finite h.d.}
Let $\gamma=\geodesic{u}{v}$ and $\gamma'=\geodesic{u'}{v'}$ be cobounded Teichm\"uller geodesics such that $\{u,v\}\cap \{u',v'\}=\emptyset$. Then for any $r\geq 0$, the set $\{\sigma\in \cT\mid d(\sigma,\gamma),d(\sigma,\gamma')\leq r\}$ is bounded. In particular, $\gamma$ and $\gamma'$ cannot be at finite Hausdorff distance.
\end{lem}

We define the \emph{canonical marked hyperbolic surface bundle} $\cS\rightarrow \cT$ as follows. The fibre $S_\sigma$ is the surface $S$ equipped with a hyperbolic metric corresponding to $\sigma\in \cT$, chosen so that $S_\sigma$  varies smoothly with $\sigma$.
\emph{The canonical hyperbolic plane bundle} $p:\cH\rightarrow \cT$ is the universal cover of $\cS$, where each fibre  is a hyperbolic plane identified with the universal cover of each $S_\sigma$.  In fact,  $\cH$ can be identified with the Teichm\"uller space of the punctured surface $S\setminus p$, and the mapping class group of $S\setminus p$ admits a fibre-preserving isometric action on $\cH$. Let $D_\sigma=p^{-1}(\sigma)$ denote the fibre at $\sigma\in \cT$.

A \emph{line} $\ell\subseteq \cT$ is the image of a  piecewise geodesic map $\mathbb{R}\rightarrow \cT$. We can pull back the canonical hyperbolic plane bundle to get a bundle $\cH_\ell\rightarrow \ell$. Each $\cH_\ell$ can be naturally endowed with a piecewise Riemannian metric, see for instance \cite{farbmosher02convexcocompact}. A line is said to be \emph{hyperbolic} if  $\cH_\ell$ is (Gromov) hyperbolic. We make use of the following theorem, proven independently by Bowditch and Mosher:

\begin{thm}[\cite{mosher03stable},\cite{bowditch13stacks}]\label{thm:mosher stable teich}
Let $Z\subseteq \cT$ be  cobounded and let $\delta\geq 0$. There is a constant $N$ such that the following holds. 
If $\ell\subseteq Z$ is a line such that $\cH_\ell$ is $\delta$-hyperbolic, then there is a cobounded Teichm\"uller geodesic $\gamma_\ell$ such that $d_{\Haus}(\gamma_\ell,\ell)\leq N$.
Conversely, if a  line $\ell\subseteq Z$ is Hausdorff equivalent to a cobounded geodesic $\gamma$, then the bundle $\cH_\ell$ is hyperbolic.  
\end{thm}

Given a cobounded geodesic $\gamma\subseteq \cT$, the bundle $\cH_\gamma\rightarrow \gamma$ also admits a singular $\Solv$ metric, which we denote $\cH_\gamma^\Solv$. If $(\cF_x,\cF_y)$ is a pair of transverse measured foliations such that $\gamma=\geodesic{[\cF_x]}{[\cF_y]}$ and $\gamma$ is parametrised by $t$, this $\Solv$     
metric is defined by $ds^2= e^{-2t}dx^2+e^{2t}dy^2+dt^2$, where $dx$ and $dy$ are the transverse measures associated to $\widetilde{\cF_x}$ and $\widetilde{\cF_y}$. It is shown in Proposition 4.4 of \cite{farbmosher2002surfacebyfree} that the identity map $\cH_\gamma\rightarrow \cH_\gamma^\Solv$ is a quasi-isometry.

Let $\Isom(\cH_\gamma^\Solv)$ denote the group of isometries  of $\cH_\gamma^\Solv$. It is shown in \cite{farbmosher2002surfacebyfree} that isometries of $\cH_\gamma^\Solv$ are \emph{fibre-preserving}, i.e. if $\phi:\cH_\gamma^\Solv\rightarrow \cH_\gamma^\Solv$ is an isometry then $\phi(D_\sigma)=D_{\overline\phi (\sigma)}$ for every $\sigma\in \gamma$, where $\overline \phi:\gamma\rightarrow \gamma$ is an induced isometry of $\gamma$.
Thus there is  a short exact sequence \[1\rightarrow \Isom_h(\cH_\gamma^\Solv)\rightarrow \Isom(\cH_\gamma^\Solv)\xrightarrow{\phi\mapsto \overline{\phi}} C_\gamma \rightarrow 1\] where $\Isom_h(\cH_\gamma^\Solv)$ denotes the kernel of the induced action on $\gamma$ and $C_\gamma$ is the associated quotient. It follows from \cite{farbmosher2002surfacebyfree} that $C_\gamma$ is isomorphic to either $1$, $\bZ_2$, $\bZ$ or $D_\infty$ and the restriction of $\Isom_h(\cH_\gamma^\Solv)$ to any fibre $D_\sigma$ is a cocompact group of isometries of $D_\sigma\cong \mathbb{H}^2$ that contains $\pi_1(S)$ as a finite index subgroup.

If $\gamma$ and $\gamma'$ are two Teichm\"uller geodesics, a quasi-isometry $f:\cH_{\gamma}\rightarrow \cH_{\gamma'}$ is said to be  \emph{fibre-preserving} if there is a constant $B$ such that for each $\sigma\in \gamma$, there is some $\mu\in \gamma'$ with $d_\Haus(f(D_\sigma),D_\mu)\leq B$.

\begin{prop}[{\cite[Lemma 6.3]{farbmosher2002surfacebyfree}}]\label{prop:axes preserved}
Let $\gamma$ and $\gamma'$ be cobounded Teichm\"uller geodesics in $\cT$ such that $\gamma$ is the axis of a pseudo-Anosov.  If there exists a fibre-preserving quasi-isometry $f:\cH_{\gamma}\rightarrow \cH_{\gamma'}$, then $\gamma'$ is also the axis of a pseudo-Anosov and there exists an isometry $\cH_{\gamma}^\Solv\rightarrow \cH_{\gamma'}^\Solv$.
\end{prop}

Much of the theory of mapping class group and Teichm\"uller space of surfaces extends to the setting of orbifolds, as observed in \cite{farbmosher02convexcocompact} and \cite{farbmosher2002surfacebyfree}. Throughout, we will restrict our attention to closed hyperbolic $2$-orbifolds with only cone singularities, since these are the only orbifolds that admit pseudo-Anosov homeomorphisms. Let $\cO$ be such an orbifold and let $\pi_1(\cO)$ denote the orbifold fundamental group. As shown in \cite{farbmosher02convexcocompact}, an analogue of the Dehn--Nielsen--Baer theorem holds so that $\Out(\pi_1(\cO))\cong \MCG(\cO)$. However, it is not the case that $\Aut(\pi_1(\cO))\cong \MCG(\cO\setminus \{p\})$. We define $\widetilde{\mathrm{MCG}}(\mathcal{O})$ to be the group of all lifts of homeomorphisms of $\cO$ to the universal cover $\widetilde{\cO}$, modulo the connected component of the identity.
In \cite{farbmosher02convexcocompact}, it was shown that there is an isomorphism of short exact sequences as follows:
\begin{equation}
\begin{tikzcd}
1 \arrow[r]   & \pi_1(\cO)\arrow[r] \arrow[d,"="] &\widetilde{\MCG}(\cO)\arrow[r,"\pi"]\arrow[d,"\cong"]&\MCG(\cO)\arrow[r]\arrow[d,"\cong"]& 1\\
1\arrow[r]&\pi_1(\cO)\arrow[r]&\Aut(\cO)\arrow[r]&\Out(\cO)\arrow[r]& 1
\end{tikzcd}
\end{equation}

\subsection*{Quasi-symmetric homeomorphisms of the circle}

\begin{defn}
Let $M\geq 1$. A homeomorphism $h:S^1\rightarrow S^1$ is said to be \emph{$M$-quasi-symmetric} if for all $x\in \bR$ and $t\in (0,2\pi)$, we have  \[\frac{1}{M}\leq \Big\lvert \frac{h(e^{i(x+t)})-h(e^{ix})}{h(e^{ix})-h(e^{i(x-t)})}\Big\rvert\leq M.\] We say that a homeomorphism   $h:S^1\rightarrow S^1$ is \emph{quasi-symmetric} if it is $M$-quasi-symmetric for some $M\geq 1$.
\end{defn}

The set of quasi-symmetric homeomorphisms of $S^1$ is closed under  composition and taking inverses. Hence quasi-symmetric homeomorphisms of $S^1$ form a group which we denote $\QSym$.
The following theorem is essentially due to Ahlfors--Beurling \cite{beurlingahlfors56boundary}, see also \cite{paulin96boundary}.

\begin{thm}[\cite{beurlingahlfors56boundary,paulin96boundary}]\label{thm:qsymvsqi}
There is an isomorphism $\Lambda:\QIsom(\Hyp)\cong \QSym$ such that the following hold:
\begin{enumerate}
\item For $K\geq 1$ and $A\geq 0$, there is an $M=M(K,A)$ such that every $(K,A)$-quasi-isometry $f:\Hyp\rightarrow \Hyp$ induces an $M$-quasi-symmetric homeomorphism $\Lambda(f):S^1\rightarrow S^1$.
\item If $f,g\in \QIsom(\Hyp)$ such that $\sup_{x\in \Hyp}d(f(x),g(x))<\infty$, then $\Lambda(f)=\Lambda(g)$.
\item If $M\geq 1$, there are $K=K(M)$ and $A=A(M)$ such that every $M$-quasi-symmetric homeomorphism $\lambda:S^1\rightarrow S^1$ induces a $(K,A)$-quasi-isometry $f:\Hyp\rightarrow\Hyp$ such that $\Lambda(f)=\lambda$.
\item If $f,g\in \QIsom(\Hyp)$ are  $(K,A)$-quasi-isometries such that $\Lambda(f)=\Lambda(g)$, then  $\sup_{x\in \Hyp}d(f(x),g(x))\leq B$ for some $B=B(K,A)$.
\end{enumerate}
\end{thm}

We fix some cocompact lattice $\Gamma\leq \Isom(\Hyp)$. Note that $\Gamma$ is a hyperbolic group with boundary $S^1$. Thus every automorphism $\Gamma\rightarrow \Gamma$ canonically induces a quasi-symmetric homeomorphism $S^1\rightarrow S^1$. This gives a  homomorphism $\theta:\Aut{\Gamma}\rightarrow \QSym$. The following lemma is implicit in Nielsen's work on surface group automorphisms.

\begin{lem}\label{lem:MCGinjective}
The map $\theta:\Aut{\Gamma}\rightarrow \QSym$ is injective.
\end{lem}
\begin{proof}
Let $\alpha\in \Aut{\Gamma}$ be an automorphism such that the induced map $\theta(\alpha)$ is trivial. Fix some $g\in \Gamma$.
For each $h\in \Gamma$, let $C_h:\Gamma\rightarrow \Gamma$ be the inner automorphism $k\mapsto hkh^{-1}$. Applying $\theta$ to the identity \begin{align}\label{eqn:identityconj} C_{\alpha(g)}\circ \alpha \circ C_{g^{-1}}\circ \alpha^{-1}=\id_G,\end{align} we deduce that $\theta(C_{\alpha(g)g^{-1}})=\id_{S^1}$ for all $g\in \Gamma$. Hence  $\alpha(g)g^{-1}$ is an isometry of $\Hyp$ which extends to a map that fixes the boundary. However any isometry of $\Hyp$ that induces the identity map on the boundary must be trivial. Therefore $\alpha(g)=g$ for all $g\in \Gamma$.
\end{proof}

We thus identify  $\Aut{\Gamma}$ with a subgroup of $\QSym$, which we also denote $\Aut{\Gamma}$. Moreover, as $\Gamma$ has trivial centre, it may also be identified with a subgroup of  $\Aut{\Gamma}\leq\QSym$.

\begin{lem}\label{lem:homeodeterminedbyconj}
If there exists an $f\in \QSym$ such that $f\circ g\circ f^{-1}=g$ for every $g\in \Gamma\leq \QSym$, then $f=\mathrm{id}_{S^1}$.
\end{lem}
\begin{proof}
Let $g\in \Gamma$ and $x\in \Fix{g}\subseteq S^1$. As $f(x)=f(g(x))=(f\circ g\circ f^{-1})(f(x))=g(f(x))$, we deduce that  $f(x)\in \Fix{g}$.

Let \[\Lambda\coloneqq \{x\in S^1\mid x\in \Fix{g}  \text{ for some loxodromic }g\in \Gamma\}.\] Since $\Lambda$ is a dense subset of $S^1$, it is sufficient to show that $f|_\Lambda=\id_\Lambda$. For any $x\in \Lambda$, let $g,h\in \Gamma$ be loxodromic elements with $x\in \Fix{g}$ and $\Fix{g}\cap \Fix{h}=\emptyset$. Suppose $\Fix{g}=\{x,y\}$ and $\Fix{h}=\{w,z\}$. Replacing $g$ with $g^{-1}$ if necessary, we may suppose that $x$ is the attracting fixed point of $g$.  We note that $\Fix{g^nhg^{-n}}=\{g^n w,g^n z\}$. Since $x$ is the attracting fixed point of $g$ and $w,z\neq y$, it follows that $g^nw\rightarrow x$ and $g^nz\rightarrow x$ as $n\rightarrow \infty$. By the above observation that $f(g^nw)=g^n w$ or $g^n z$, we see that $f(g^nw)\rightarrow x$ as $n\rightarrow \infty$, thus $f(x)=x$ by continuity of $f$.
\end{proof}

\begin{cor}\label{cor:homeodeterminedbyconj}
The inclusion \[\theta:\Aut{\Gamma}\rightarrow \Norm_{\QSym}(\Gamma)\coloneqq \{f\in \QSym\mid f\Gamma f^{-1}=\Gamma\}\] is surjective.
\end{cor}
\begin{proof}
Suppose there exists an $f\in \QSym$ with $f\Gamma f^{-1}=\Gamma$. There is an automorphism $\alpha\in \Aut{\Gamma}$ such that for all $g\in \Gamma$, $f \theta(g) f^{-1}=\theta(\alpha(g))$. By identity (\ref{eqn:identityconj}) in Lemma \ref{lem:MCGinjective}, we see that $\theta(\alpha(g))=\theta(\alpha)\theta(g)\theta(\alpha^{-1})$. Hence $\theta(\alpha)^{-1}f \theta(g) f^{-1}\theta(\alpha)=\theta(g)$ for all $g\in \Gamma$. Thus Lemma \ref{lem:homeodeterminedbyconj} ensures that $\theta(\alpha)=f$.
\end{proof}

\subsection*{Proof of quasi-isometric rigidity of surface group extensions}
We fix a closed hyperbolic surface $S$ and a finitely generated, infinite ended, irreducible subgroup $L\leq \MCG(S)$ of type $F_3$. Let $\Gamma_L$ be the group extension \[1\rightarrow \pi_1(S)\rightarrow \Gamma_L\xrightarrow{\pi} L \rightarrow 1.\] By Brown's criterion \cite{brown1987finiteness}, $\Gamma_L$ is also of type $F_3$. (See also Exercise 1 of Section 7.2 in \cite{geoghegan2008topological}.)
We now construct a geometric model of the group $\Gamma_L$, the details of which are found in \cite{farbmosher02convexcocompact}. 

We first fix a generating set of $L$ and embed the correspond Cayley graph $X_L$ into $\cT$ as follows. We identify the vertex set of $X_L$ with  some $L$-orbit in $\cT$. For each edge in $X_L$, we join the corresponding points of $\cT$ by a piecewise  geodesic, all of which are disjoint except at their endpoints. This is done in an equivariant way. The action of $L$ on $X_L$ is cocompact and properly discontinuous. Moreover, the inclusion $X_L\rightarrow\cT$ is a coarse  embedding. For ease of notation, we identify $X_L$ with its image in $\cT$. We now pullback the canonical hyperbolic plane bundle to obtain the bundle $p_L:\cH_L\rightarrow X_L$. As in \cite{farbmosher02convexcocompact},  we endow $\cH_L$ with a piecewise Riemannian metric such that  $\Gamma_L$ acts properly discontinuously and cocompactly on $\cH_L$.

 Given a pseudo-Anosov $g\in L$, we say that a line $\ell_g\subseteq \cT$ is a \emph{coarse axis}  of $g$ if there is a bounded set $\Omega\subseteq \cT$ such that $\ell_g\subseteq \langle g\rangle \Omega$. Coarse axes of pseudo-Anosovs always exist and are coarsely well-defined, i.e. if $\ell_g$ and  $\ell'_g$ are  coarse axes of the same pseudo-Anosov $g$, then $d_\Haus(\ell_g,\ell'_g)<\infty$.

\begin{prop}\label{prop:pA in subgroup}
Let  $\gamma\subseteq \cT$ be the axis of a pseudo-Anosov $g\in \MCG(S)$ such that $\gamma\subseteq N_B(X_L)$ for some $B\geq 0$.  Then $g^n\in L$ for some $n$ sufficiently large.
\end{prop}
\begin{proof}
For any $x\in X_L$,  $d_\Haus(\langle g\rangle\cdot x,\gamma)<\infty$ so that $\langle g\rangle\cdot x \subseteq N_C(L\cdot x)$ for some $C\geq 0$. Since the map $\MCG(S)\rightarrow \cT$ given by $f\mapsto f\cdot x$ is a coarse embedding, we see that there is some $B'$ such that $\langle g\rangle \subseteq N_{B'}(L)$ in $\MCG(S)$. Thus $g^n\in  L$ for some $n$ sufficiently large by \cite[Corollary 2.14]{mosher2011quasiactions}. 
\end{proof}

Our first step in proving Theorem \ref{thm:qirigidity surface group extension} is showing that quasi-isometries of $\cH_L$ are fibre-preserving. This follows from Theorem \ref{thm:main fib bundle} and Lemma \ref{lem:fibre-preserving qi imp qi between base spaces}. This is the only point where we use the fact that $L$ is infinite ended and of type $F_3$.
\begin{prop}\label{prop:surface fibpreserve}
For every $K\geq 1$ and $A\geq 0$, there exist constants $B\geq 0$,  $K'\geq 1$ and $A'\geq 0$ such that the following holds. For every $(K,A)$-quasi-isometry $f:\cH_L\rightarrow \cH_L$, there exists  a $(K',A')$-quasi-isometry $\hat f:X_L\rightarrow X_L$ such that for every $\sigma\in X_L$, $d_\Haus(f(D_\sigma),D_{\hat{f}(\sigma)})\leq B$.
\end{prop}
Each quasi-isometry $f:\cH_L\rightarrow \cH_L$ induces a map $\theta(f)\in \QSym$ as follows. We fix a fibre $D_0\subseteq \cH_L$. By Proposition \ref{prop:surface fibpreserve}, $f(D_0)$ has finite Hausdorff distance from some fibre $D_x\subseteq \cH_L$. However, since $D_x$ has finite Hausdorff distance from $D_0$, there is a quasi-isometry $f_0:D_0\rightarrow D_0$ such that $\sup_{y\in D_0} d(f_0(y),f(y))<\infty$. This map is coarsely well-defined, so induces a well-defined map $\theta(f)\in \QSym$, where $S^1$ is identified with the Gromov boundary of $D_0$. This map is independent of the choice of fibre $D_0$, so  $\theta:\QI(\cH_L)\rightarrow \QSym$ is a well-defined homomorphism.

\begin{prop}\label{prop:qigp injective}
For any $K\geq 1$ and $A\geq 0$, there exists a  $B\geq 0$ such that whenever $f:\cH_L\rightarrow \cH_L$ is a $(K,A)$-quasi-isometry with $\theta(f)=\id_{S^1}$, $\sup_{x\in \cH_L}d(x,f(x))\leq B$. In particular, $\theta:\QI(\cH_L)\rightarrow \QSym$ is injective.
\end{prop}

To prove this, we make use of another result from \cite{farbmosher2002surfacebyfree}.

\begin{prop}\label{prop:endpts of leaves preserved}
Let $g,g'\in L$ be two pseudo-Anosovs with coarse axes $\ell_g ,\ell_{g'}\subseteq X_L$. Suppose that  $f:\cH_L \rightarrow \cH_L$ is a quasi-isometry such that $d_\Haus(\hat f(\ell_{g}),\ell_{g'})<\infty$, with $\hat f$ as in Proposition \ref{prop:surface fibpreserve}. Then $\theta(f)(E(g))=E(g')$.
\end{prop}
\begin{proof}
Let $\gamma$ be the axis of $g$ and let $\cF_x$ and $\cF_y$ be the associated transverse measured foliations.  Then the bundle $\cH_\gamma^\Solv$ has   \emph{stable and unstable foliations} as defined in Section 5 of \cite{farbmosher2002surfacebyfree}. These are defined so that the intersection of the stable (resp. unstable) foliation with each fibre $D_{\sigma}$ is  $\widetilde{\cF_y}$ (resp. $\widetilde{\cF_x}$). If $\gamma'$ is  the axis of $g'$, then  $\cH_{\gamma'}^\Solv$ also has stable and unstable foliations.

Since $d_\Haus(\hat f(\ell_{g}),\ell_{g'})<\infty$, there is a fibre-preserving quasi-isometry $f_\gamma:\cH_{\gamma}^\Solv\rightarrow \cH_{\gamma'}^\Solv$ such that $\sup_{x\in \cH_{\gamma}^\Solv}d(f_\gamma(x),f(x))<\infty$.
 Proposition 5.2 of \cite{farbmosher2002surfacebyfree} then says that $f_\gamma$ coarsely respects the stable and unstable foliations. In particular, by restricting to a fibre of $\cH_{\gamma}^\Solv$ and using the definition of $\theta(f)$, we deduce that $\theta(f)(E(g))=E(g')$ as required.
\end{proof}

We also need the following elementary connect-the-dots argument:
\begin{lem}\label{lem:image of line}
For every $K\geq 1$ and $A\geq 0$, there is a constant $B$ such that for every $(K,A)$-quasi-isometry $f:X_L\rightarrow X_L$ and every line $\ell\subseteq X_L$,  there is a line $\ell'\subseteq X_L$ with $d_\Haus(f(\ell),\ell')\leq B$. 
\end{lem}
\begin{proof}
As $\ell$ is a line, there is a piecewise geodesic map $r:\bR\rightarrow \cT$ with image $\ell$. By reparametrising $\bR$, we may assume that $r$ is $1$-Lipschitz. We define a map $r':\bR\rightarrow \cT$ such that for all $n\in \bZ$, $r'|_{[n,n+1]}$ is a reparametrised Teichm\"uller geodesic segment in $X_L$ from $f(r(n))$ to $f(r(n+1))$. Thus the image of $r'$ is a line $\ell'\subseteq X_L$  with $d_\mathrm{Haus}(f(\ell),\ell')\leq B$ for some suitable $B$.
\end{proof}

\begin{proof}[Proof of Proposition \ref{prop:qigp injective}]
Suppose $f:\cH_L\rightarrow \cH_L$ is a $(K,A)$-quasi-isometry such that $\theta(f)=\id_{S^1}$. Let $\hat{f}:X_L\rightarrow X_L$ be the induced quasi-isometry as in Proposition \ref{prop:surface fibpreserve}, which can be taken to be a $(K',A')$-quasi-isometry for some $K'\geq 1$ and $A'\geq 0$ depending only on $K$ and $A$. 
Let $g\in L$ be a pseudo-Anosov with coarse axis $\ell_g\subseteq X_L$ and axis $\gamma_g$. By Lemma \ref{lem:image of line}, there is a line $\ell'_g\subseteq \cT$ such that $d_\mathrm{Haus}(\hat f(\ell_g),\ell_g')\leq B$ for some $B=(K',A')$. Since $\ell_g$ is a coarse axis and so has finite Hausdorff distance from the geodesic $\gamma_g$, Theorem \ref{thm:mosher stable teich} says that $\cH_{\ell_g}$ is $\delta$-hyperbolic for some $\delta$.   
Thus $f$ induces a quasi-isometry $\cH_{\ell_g}\rightarrow \cH_{\ell'_g}$, so that $\cH_{\ell'_g}$ is $\delta'=\delta'(K,A,\delta)$-hyperbolic.  Theorem \ref{thm:mosher stable teich} implies $\ell'_g$ has finite Hausdorff distance from some Teichm\"uller geodesic $\gamma'_g$. Since $f$ induces a fibre-preserving quasi-isometry $\cH_{\gamma_g}\rightarrow \cH_{\gamma'_g}$, Proposition \ref{prop:axes preserved} then says that $\gamma'_g$ is the axis of a pseudo-Anosov $g'\in \MCG(S)$. By Proposition \ref{prop:pA in subgroup},  we can assume that $g'\in L$ and so $\ell'_{g}$ is a coarse axis of $g'$. Thus Proposition \ref{prop:endpts of leaves preserved} says that $E(g')=\theta(f)(E(g))=E(g)$ and so Lemma \ref{lem:endpts of leaves disjoint} ensures that $g$ and $g'$ share an axis. Thus $\gamma_g=\gamma'_g$ and so $d_\Haus(\ell_g,\widehat f(\ell_g))<\infty$.

 We now fix some conjugacy class $\cC$ of pseudo-Anosovs in $L$ and equivariantly pick coarse axes $\{l_g\subseteq X_L\mid g\in \cC\}$. Thus there is a $\delta$ such that $\cH_{\ell_g}$  is $\delta$-hyperbolic for every $g\in \cC$. Hence there is $\delta'=\delta'(\delta,K,A)$ and $B$ such that for all $g\in \cC$, there is a line $\ell'_g$ such that $d_\mathrm{Haus}(\hat f(\ell_g),\ell'_g)\leq B$ and $\cH_{\ell'_g}$ is $\delta'$-hyperbolic. Theorem \ref{thm:mosher stable teich} then says that there is some $C=C(\delta,\delta', L)$ such that for all $g\in \cC$, both $\ell_g$ and $\ell'_g$ have Hausdorff distance at most $C$ from a Teichm\"uller geodesic. By the argument in the previous paragraph and Lemma \ref{lem:finite h.d.}, we  deduce $\ell_g$ and $\ell'_g$ have Hausdorff distance at most $C$ from the same Teichm\"uller geodesic, and so there is a $D\geq 0$ such that $d_\Haus(\ell_g,\hat f(\ell_g))\leq D$ for every $g\in \cC$.

Let $g,h\in \cC$ be two pseudo-Anosovs such that $\Fix{g}\cap \Fix{h}=\emptyset$. We fix some bounded $\Omega\subseteq \cT$ such that $X_L\subseteq L\Omega$ and pick $r$ large enough so that $d(\omega,\ell_g),d(\omega,\ell_h)\leq r$ for all $\omega\in \Omega$.  By Lemma \ref{lem:finite h.d.} there is a constant $R$ such that if $d(\sigma,\ell_g),d(\sigma,\ell_h)\leq K'r+A'+D$ for some $\sigma\in \cT$, then $d(\sigma,\omega)\leq R$ for all $\omega\in \Omega$. 
Let $\sigma\in X_L$. Pick $k\in L$ and $\omega\in \Omega$ such that $\sigma=k\omega$. Thus $d(\sigma,k\ell_g), d(\sigma,k\ell_h)\leq r$. Notice that $k\ell_g=\ell_{kgk^{-1}}$ and $k\ell_h=\ell_{khk^{-1}}$.  Then $d(\hat f(\sigma),k\ell_g),d(\hat f(\sigma),k\ell_h)\leq K'r+A'+D$ so that $d(\hat f(\sigma),\sigma)\leq R$ for all $\sigma\in X_L$.

It follows from the definition of $\hat f$ that there is a constant $E=E(K,A,R)$ such that for every fibre $D_\sigma\subseteq \cH_L$, $d_\Haus(D_\sigma,f(D_\sigma))\leq E$. Proposition \ref{prop:fibreqi induces qi between fibres} tells us that there exists $K''\geq 1$ and $A''\geq 0$ such that for every $\sigma\in X_L$, there is a $(K'',A'')$-quasi-isometry $f_\sigma:D_\sigma\rightarrow D_\sigma$ such that $\sup_{x\in D_\sigma}d(f_\sigma(x),f(x))\leq E$. Since $\theta(f)=\id_{S^1}$, the induced boundary homeomorphism $\partial f_\sigma:\partial D_\sigma\rightarrow \partial D_\sigma$ is the identity. Hence by Theorem \ref{thm:qsymvsqi}, there is an $F=F(K'',A'')$ such that $\sup_{x\in D_\sigma}d(x,f_\sigma(x))\leq F$. Thus $\sup_{x\in \cH_L}d(x,f(x))\leq E+F$. Notice that the constants $E$, $F$ depend only on the constants $K$ and $A$. (They also depend on the choice of conjugacy class $\cC$, the elements $g,h\in \cC$ and the choice of coarse axes, but  these can be fixed as we  vary $f$ over all $(K,A)$-quasi-isometries.)
\end{proof}

Let $g\in L$ be a pseudo-Anosov with axis $\gamma$. Recall that $\Isom_h(\cH_\gamma^\Solv)$ is the subgroup of isometries of $\cH_\gamma^\Solv$ that preserve each fibre. For ease of notation, let us denote $\Isom_h(\cH_\gamma^\Solv)$ by $\Pi_g$. Choosing some $\sigma\in \gamma$,  $\Pi_g$ restricts to a group of isometries $\Isom(D_\sigma)$. Since $D_\sigma$ has finite Hausdorff distance from $D_0$, we have an induced  homomorphism $\Pi_g\rightarrow \QSym$, where $S^1$ is identified with the Gromov boundary of $D_0$. Any isometry of $D_\sigma\cong \mathbb{H}^2$ that extends to the identity map of $S^1$ must be the identity on $D_\sigma$. Thus $\Pi_g\rightarrow \QSym $ is injective and so $\Pi_g$ can be regarded as a subgroup of $\QSym$. 

Let $\Pi\coloneqq \cap \{\Pi_g\mid g\in L \text{ is a pseudo-Anosov}\}.$ We recall that each $\Pi_g$ contains $\pi_1(S)$ as a finite index subgroup. Hence $\Pi$ can be identified with a subgroup of $\Isom(\Hyp)$ that contains $\pi_1(S)$ as a finite index subgroup. Let $\cO$ denote the quotient orbifold $\Pi\bs\Hyp$ with orbifold fundamental group $\Pi$. Then there is a finite cover $q:S\rightarrow \cO$ which induces the  inclusion $q_*:\pi_1(S)\rightarrow\Pi$. As remarked in \cite{farbmosher2002surfacebyfree} and \cite{mosher2003fiber}, the preceding theory generalises to the  mapping class group and Teichm\"uller space of $\cO$. Let $\cT_\cO$ denote the Teichm\"uller space of $\cO$.

\begin{lem}\label{lem:sbgp descend}
Recall the surface group extension $\Gamma_L$ can be naturally identified with a subgroup of $\Aut(\pi_1(S))$. Then there is a unique map $\phi:\Gamma_L\rightarrow \Aut(\Pi)$ such that:
\begin{enumerate}
\item $\phi(\alpha)|_{\pi_1(S)}=\alpha$;
\item there is a  monomorphism $\overline \phi: L\hookrightarrow \Out(\Pi)$   such that the following diagram commutes
\[\begin{tikzcd}
\Gamma_L \arrow[r,"\pi"]\arrow[d,"\phi"] &L\arrow[hookrightarrow]{d}{\overline \phi}\\
\Aut(\Pi)\arrow[r,"\pi'"] & \Out(\Pi)
\end{tikzcd}\]
where $\pi$ and $\pi'$ are quotient maps.

\end{enumerate}
\end{lem}

\begin{proof}
We first note that Lemma \ref{lem:MCGinjective} and Corollary \ref{cor:homeodeterminedbyconj} imply that any automorphism of $\Gamma$ that restricts to an automorphism of $\pi_1(S)$ is determined by its restriction. Thus if $\phi$ exists, it is unique.
Suppose $h\in \Gamma_L$ and $g\in L$ is a pseudo-Anosov with axis $\gamma$. Then $\pi(h)\gamma$ is the axis of the pseudo-Anosov $\pi(h)g\pi(h)^{-1}$. Thus $h$ restricts to a fibre-preserving isometry $\cH_\gamma^\Solv\rightarrow \cH_{\pi(h)\gamma}^\Solv$, hence $h$ conjugates $\Pi_g$ to $\Pi_{\pi(h)g\pi(h)^{-1}}$. This holds for every pseudo-Anosov $g\in L$, so we deduce that $h$ defines an automorphism of $\Pi$ that we denote $\phi(h)$.

Suppose there exist elements $g,h\in \Gamma_L$ such that $\pi(g)=\pi(h)$. Then $g^{-1}h$ is an inner automorphism of $\pi_1(S)$, so clearly extends to an inner automorphism of $\Pi$. Thus $\pi'\circ\phi(g)=\pi'\circ\phi(h)$, hence $\phi$ induces a well-defined homomorphism $\overline{\phi}:L\rightarrow \Out(\Pi)$ such that $\pi'\circ \phi=\overline \phi\circ \pi$. 

Finally, we claim that $\overline{\phi}$ is injective. 
Suppose that $\pi(h)\in \ker(\overline{\phi})$. Then $\phi(h)$ is an inner automorphism of $\Pi$, so via the the Dehn--Nielsen--Baer theorem, corresponds to an orbifold homeomorphism $B$ of $\cO$ that is isotopic to the identity. We can  lift this isotopy to an isotopy of $S$ from the identity to a lift $A$ of $B$. Thus $A$ is a homeomorphism of $S$ that represents the trivial mapping class. Since $A$ is a lift of $B$, it follows that $A$ is a homeomorphism of $S$ that represents $h$, and so $h$ must be an inner automorphism of $\pi_1(S)$ via the the Dehn--Nielsen--Baer theorem.
\end{proof}

Lemma \ref{lem:sbgp descend} has a topological interpretation. Via the Dehn--Nielsen--Baer theorem, the image of $\ol \phi$ can be naturally identified with a subgroup of $\MCG(\cO)$. It follows from Lemma \ref{lem:sbgp descend} and the Dehn--Nielsen--Baer theorem that  for each $l\in L$, there are homeomorphisms $A$ of $S$ and $B$ of $\cO$ representing $l$ and $\ol \phi(l)$ respectively, such that $A$ is a lift of $B$. In the terminology of \cite{mosher2003fiber}, Lemma \ref{lem:sbgp descend} says that $L\leq \MCG(S)$ \emph{descends to a subgroup of $\MCG(\cO)$}. It is shown in \cite{mosher2003fiber} that $\cO$ is the smallest orbifold such that $L$ descends.

We let $L'\coloneqq \im(\overline \phi)\cong L$ and $\Gamma_{L'}\coloneqq\pi'^{-1}(L')$. As  the inclusion $\pi_1(S)\rightarrow \Pi$ has finite index image, we see that $\Gamma_L$ maps to a  finite index subgroup of $\Gamma_{L'}$, which we identify with $\Gamma_L$. We observe that there is an embedding $\cT_\cO\rightarrow \cT_S=\cT$. This is because a hyperbolic structure on $\cO$ can be lifted to a hyperbolic structure on $S$. (Alternatively, a discrete faithful representation $\Pi\rightarrow \PSL(2,\bR)$ restricts to a discrete  faithful representation $\pi_1(S)\rightarrow \PSL(2,\bR)$.) Teichm\"uller's theorem implies that this in fact an isometric embedding, since Teichm\"uller geodesics in $\cT_\cO$ map to Teichm\"uller geodesics in $\cT_S$. We thus identify $\cT_\cO$ with a subset of $\cT$. Since $L'\leq \Out(\Pi)\cong \MCG(\cO)$, we can always ensure that $X_L$, as defined  above, is a subset of $\cT_\cO$. The canonical hyperbolic plane bundle $\cH_\cO\rightarrow \cT_\cO$ associated to $\cO$ is thus a pullback of the canonical hyperbolic plane bundle $\cH\rightarrow \cT$ associated to $S$.

  Let $M=\Comm_{\MCG(\cO)}(L')$ and let $\Gamma_M$ be the corresponding subgroup of $\Aut(\Pi)\cong \widetilde \MCG(\cO)$. For every  $g\in M$,   $d_\Haus(L',gL')<\infty$ by Proposition \ref{prop:coset characterisation}, from which it follows that $d_\Haus(\Gamma_{L'},h\Gamma_{L'})<\infty$ for all $h\in \Gamma_M$. Since $\Gamma_L$ is a finite index subgroup of $\Gamma_{L'}$, we see that $d_\Haus(\Gamma_{L},h\Gamma_{L})<\infty$ for all $h\in \Gamma_M$. Thus for each $h\in \Gamma_M$, left multiplication by $h$ followed by closest point projection defines a quasi-isometry  $\Gamma_{L}\rightarrow \Gamma_{L}$. This gives a homomorphism $\Psi:\Gamma_M\rightarrow \QIsom(\Gamma_L)$. We  classify all quasi-isometries of $\Gamma_L$ as follows:

\begin{thm}
The map $\Psi:\Gamma_M\rightarrow \QIsom(\Gamma_L)$ is an isomorphism.
\end{thm}
For injectivity we note that  Lemma \ref{lem:MCGinjective} ensures two distinct elements $g,h\in\Gamma_M\leq \widetilde \MCG(\cO)\cong \Aut(\Pi)$ induce different elements of $\QSym$, hence Proposition \ref{prop:qigp injective} ensures that $\Psi(g)\neq \Psi(h)$. Surjectivity follows immediately  from the following lemma.

\begin{lem}\label{lem:quant mos fibre preserve}
Let $K\geq 1$ and $A\geq 0$. Then there is a constant $C=C(K,A)$ such that whenever $f$ is a $(K,A)$-quasi-isometry $f:\Gamma_{L}\rightarrow \Gamma_{L}$, there exists a $\psi_f\in \Gamma_M$ such that for all $x\in \Gamma_L$, \[d_{\widetilde{\mathrm{MCG}}(\mathcal{O})}(\psi_f x,f(x))\leq C.\]
\end{lem}
\begin{proof}

The argument is similar to the proof of Proposition \ref{prop:qigp injective}. It is enough to show the following: there is a constant $C=C(K,A)$ such that whenever $f$ is a $(K,A)$-quasi-isometry $f:\cH_{L}\rightarrow \cH_{L}$, there exists a $\psi_f\in \Gamma_M$ such that for all $x\in \Gamma_L$, \[d_{\cH_\cO}(\psi_f x,f(x))\leq C.\] This is because $\Gamma_L$ acts properly discontinuously and cocompactly on $\cH_L$, and the orbit map $\widetilde \MCG(\cO)\rightarrow \cH_\cO$ is a coarse embedding.

Suppose $f:\cH_L\rightarrow \cH_L$ is a $(K,A)$-quasi-isometry. Proposition \ref{prop:surface fibpreserve} ensures there are constants $K', A'$ and $B$ such that there is a $(K',A')$-quasi-isometry $\hat f:X_L\rightarrow X_L$ such that $d_\Haus(f(D_\sigma),D_{\hat f(\sigma)})\leq B$ for all $\sigma\in X_L$.
Let $g\in L$ be a pseudo-Anosov with coarse axis $\ell_g\subseteq X_L$ and axis $\gamma_g\subseteq \cT_\cO\subseteq \cT_S$. By Lemma \ref{lem:image of line}, there is a line $\ell'_g\subseteq X_L$ such that $d_\mathrm{Haus}(\hat f(\ell_g),\ell_g')\leq C$ for some $C=(K',A')$. Since $\ell_g$ is a coarse axis, $\cH_{\ell_g}$ is $\delta$-hyperbolic for some $\delta$.   
Thus $f$ induces a quasi-isometry $f_g:\cH_{\ell_g}\rightarrow \cH_{\ell'_g}$, so that $\cH_{\ell'_g}$ is $\delta'$-hyperbolic and therefore has finite Hausdorff distance from some Teichm\"uller geodesic $\gamma'_g$. 

By Propositions \ref{prop:axes preserved} and \ref{prop:pA in subgroup}, we deduce that $\gamma'_g$ is the axis of a  pseudo-Anosov $g'\in L$ and that there is a fibre-preserving isometry $\cH_{\ell_g}\rightarrow \cH_{\ell'_g}$. This isometry conjugates $\Isom_h(\cH_{\ell_g})$ to $\Isom_h(\cH_{\ell'_{g}})$, so that $\theta(f)$ conjugates $\Pi_g$ to $\Pi_{g'}$ in $\QSym$. Since this holds for each pseudo-Anosov $g\in L$, $\theta(f)$ conjugates $\Pi$ in $\QSym$. Thus Corollary \ref{cor:homeodeterminedbyconj} and Proposition \ref{prop:qigp injective} ensures there is an element $\psi_f\in \widetilde\MCG(\cO)\cong \Aut(\Pi)$ such that $\sup_{x\in X_L}d_{\cH_\cO}(f(x),\psi_f\cdot x)<\infty$. 

To complete the proof, we need only show that $\sup_{x\in X_L}d_{\cH_\cO}(f(x),\psi_f\cdot x)$ can be bounded by a constant that depends only on $K$ and $A$ and not the quasi-isometry $f$. We argue similarly to the proof of Proposition \ref{prop:qigp injective}. We fix a conjugacy class $\cC$ of pseudo-Anosovs in $L$ and choose coarse axes $\{\ell_g\mid g\in \cC\}$ equivariantly. Thus there is a $\delta$ such that $\cH_{\ell_g}$ is $\delta$-hyperbolic for every $g\in \cC$. There is therefore a $\delta'$ such that $\cH_{\ell'_g}$ is $\delta'$-hyperbolic for every $g\in \cC$. 

Let $\phi_f\coloneqq \pi'(\sigma_f)\in L'$, where $\pi':\Aut(\Pi)\rightarrow \Out(\Pi)\cong \MCG(\cO)$ is the quotient map. As $\phi_f$ acts by isometries  on $\cT_\cO$, $\phi_f\ell_g$ is a hyperbolic line in $\cT_\cO$ such that $\cH_{\phi_f\ell_g}=\psi_f\cH_{\ell_g}$ is $\delta$-hyperbolic.  Since $\sup_{x\in X_L}d_{\cH_\cO}(f(x),\psi_f\cdot x)<\infty$, it follows that $d_\mathrm{Haus}d_{\cT_\cO}(\hat{f}(\ell_g),\phi_f(\ell_g))<\infty$.  Using Theorem \ref{thm:mosher stable teich}, there is a constant $D=D(\delta,\delta')$ such that $d_\mathrm{Haus}(\hat{f}(\ell_g),\phi_f(\ell_g))\leq D$ for each $g\in \cC$.  Furthermore, there is a constant $E$ such that every point $x\in X_L$ satisfies $x\in N_E(\ell_g)$ for some $g\in \cC$. Thus $\phi_f\cdot X_L\subseteq N_{D+E}(X_L)$ and so $\psi_f\cdot \cH_L\subseteq N_{D+E}(\cH_L)$. As usual, the constants $D$ and $E$ don't depend on $f$, only the constants $K$ and $A$

To  conclude, we define a closest point projection map $u:N_{D+E}(\cH_L)\rightarrow \cH_L$ whose quasi-isometry constants depend only on $D+E$. Thus there exist constants $K''$ and $A''$ such that for every $(K,A)$-quasi-isometry $f:\cH_L\rightarrow \cH_L$, the composition $\overline f\circ u \circ \psi_f:\cH_L\rightarrow \cH_L$ is a $(K'',A'')$-quasi-isometry, where $\overline f$ is a coarse inverse to $f$.
Thus Proposition \ref{prop:qigp injective} ensures that there is a constant $F$ such that $\sup_{\cH_L}d(\overline f\circ u \circ \psi_f(x),x)\leq F$ for every $(K,A)$-quasi-isometry $f:\cH_L\rightarrow \cH_L$. The result immediately follows.
\end{proof}

We deduce Theorem \ref{thm:qirigidity surface group extension} from Lemma \ref{lem:quant mos fibre preserve} using an argument similar to that found in \cite{schwartz95r1lattices}.

\qirsurfex
\begin{proof}
We  fix an irreducible, infinite-ended subgroup $L\leq \mathrm{MCG}(S)$ of type $F_3$. Suppose  $f:G\rightarrow \Gamma_L$ is a quasi-isometry, where $G$ is any finitely generated group. Let $\overline{f}$ be a coarse inverse to $f$ and for each $g\in G$, let $L_g:G\rightarrow G$ be left multiplication by $g$. 
Then $f_g\coloneqq f\circ L_g\circ \overline{f}$ is a quasi-isometry of $\Gamma_L$. Since $L_g$ is an isometry of $\Gamma_L$, $\{f_g\mid g\in G\}$ is a uniform set of quasi-isometries, i.e. there exists $K\geq 1$ and $A\geq 0$ such that each $f_g$ is a $(K,A)$-quasi-isometry. We may also assume that $f$ and $\overline{f}$ are $(K,A)$-quasi-isometries. 
Thus Lemma \ref{lem:quant mos fibre preserve} ensures there is a $C\geq 0$ such that for every $g\in G$ there exists an $\psi_g\in \Gamma_M$ with \begin{align}
d(\psi_g\cdot x,f_g(x))\leq C
\end{align}
for all $x\in \Gamma_L$. We claim that the map $\Psi:G\rightarrow \Gamma_M$ given by $g\mapsto \psi_g$ has finite kernel and image commensurable to $\Gamma_L\leq \Gamma_M$. This map is well-defined  and a homomorphism.

To shown $\ker(\Psi)$ is finite,  we pick $x_0\in G$ such that $d(f(x_0),e)\leq A$ and suppose that $g\in \ker(\Psi)$. Then $d(e,f_g(e))\leq C$. Thus \begin{align*}
\frac{1}{K}d(x_0,gx_0)-A&\leq d(f(x_0),f(gx_0))\\
&\leq d(f(x_0),e)+d(e,f_g(e))+d(f_g(e),f(gx_0))\\
&\leq  Kd(\overline{f}(e),x_0)+2A+C.
 \end{align*}
 There are only finitely many such $g\in G$, thus $\ker(\Psi)$ is finite.

We now show $\im({\Psi})$ has finite Hausdorff distance from  $\Gamma_L$. For each $x\in \Gamma_L$,  there exists some $k_x\in G$ with $d(f(k_x),x)\leq A$. Let $g_x\coloneqq k_x\overline{f}(e)^{-1}$. Then $f_{g_x}(e)=f(g_x\overline{f}(e))=f(k_x)$ and so 
\[d(\psi_{g_x},x)\leq d(\psi_{g_x}\cdot e,f_{g_x}(e))+d(f(k_x),x)\leq C+A,\] thus $\Gamma_L\subseteq N_{A+C}(\im{\Psi})$. Conversely for each $\psi_g\in \im{\Psi}$, $d(\psi_g,f_g(e))\leq C$, so that $\im{\Psi}\subseteq N_C(\Gamma_L)$.
\end{proof}

%% file: products.tex
\section{Groups quasi-isometric to products}\label{sec:gpsqitoprods}
In this appendix we demonstrate further instances in which Question \ref{ques:alnormal subgp qiinvariant} is true, even if Theorem \ref{thm:main intro} doesn't necessarily apply. Theorem \ref{thm:qi to product contains alnorm} is not a new result in and of itself, rather it is a combination of existing results --- namely \cite{kapovich1998derahm} and \cite{mosher2003quasi} --- formulated in the language of almost normal subgroups and quotient spaces. Theorem \ref{thm:qi to product contains alnorm}  can be used to deduce Proposition \ref{prop:gps qi to TxT}.

The  following theorem says that under some fairly general coarse non-positive curvature assumptions, quasi-isometries preserve direct products. We refer to \cite{kapovich1998derahm} for a definition of coarse type I and II.
\begin{thm}[{\cite[Theorem B]{kapovich1998derahm}}]\label{thm:dehrahm}
Suppose $X=Z\times \Pi_{i=1}^k X_i$ is a geodesic metric space such that the asymptotic cone of $Z$ is homeomorphic to $\mathbb{R}^n$ and each $X_i$ is either of coarse type I or II. Let $p_i:X\rightarrow X_i$ be the projection map. Then for every $K\geq 1$ and $A\geq 0$, there exist constants $K'\geq 1$ and $A',D\geq 0$ such that the following holds: 

Whenever $f:X\rightarrow X$ is a $(K,A)$-quasi-isometry, there is a $\sigma_f\in\mathrm{Sym}(k)$ so that for each $i$, there exists a  $(K',A')$-quasi-isometry $f_i$ such that the following diagram commutes up to error at most $D$.
\[
\begin{tikzcd}  
X\arrow{r}{f}\arrow{d}{p_i}& X\arrow{d}{p_{\sigma_f(i)}}\\
X_i\arrow{r}{f_i}& X_{\sigma_f (i)}
\end{tikzcd}
\]
\end{thm}

We recall the definition of a quasi-action:
\begin{defn}
Let $G$ be a group and $X$ be a metric space. A \emph{quasi-action} of $G$ on $X$ associates to each $g\in G$ a quasi-isometry $f_g:X\rightarrow X$ such that:
\begin{enumerate}
\item there are constants $K\geq 1$ and $A\geq 0$ such that each $f_g$ is a $(K,A)$-quasi-isometry;
\item for all $g,h\in G$ and $x\in X$, $d\big(f_h(f_g(x)),f_{hg}(x)\big)\leq A$;
\item for all $x\in X$, $d(f_e(x),x)\leq A$.
\end{enumerate}
For $(K,A)$ as above, we say that $\{f_g\}_{g\in G}$ is a $(K,A)$-quasi-action.
\end{defn}

\begin{lem}\label{lem:qa on direct factor}
Suppose a finitely generated group $G$ quasi-acts on a metric space $X=Z\times \Pi_{i=1}^k X_i$ as in Theorem \ref{thm:dehrahm}. Then for each $i$, there is a finite index subgroup $G_0\leq G$ such that for $1\leq i\leq k$, $G_0$ admits a quasi-action on $X_i$.
\end{lem}
\begin{proof}
Let $\{f_g\}_g\in G$ be the set of quasi-isometries associated to the quasi-action of $G$ on $X$. We apply Theorem \ref{thm:dehrahm} to each $f_g$. The map $f_g\mapsto \sigma_{f_g}$ defines a homomorphism $G\rightarrow \mathrm{Sym}(k)$ whose kernel we denote $G_0$. Then the set $\{(f_g)_i:X_i\rightarrow X_i\mid g\in G_0\}$ defines a quasi-action of $G_0$ on $X_i$. Note that $G=G_0$ if none of the $X_i$ are quasi-isometric to one another. 
\end{proof}

Given two quasi-actions $\{f_g\}_{g\in G}$ and $\{k_g\}_{g\in G}$ on metric spaces $X$ and $Y$, a quasi-conjugacy is a quasi-isometry $r:X\rightarrow Y$ such that there exists a constant $A\geq 0$ with $d_Y((k_g\circ r)(x),(r\circ f_g)(x))\leq A$ for all $x\in X$ and $g\in G$. A \emph{bushy} tree is one in which every point of $T$ has uniformly bounded distance from a vertex with at least three unbounded complementary components. We can now state the following theorem of Mosher--Sageev--Whyte:

\begin{thm}[\cite{mosher2003quasi}]\label{thm:mswthm1}
Let $G$ be a group admitting a cobounded quasi-action  on a bounded valence bushy tree $T$. Then $G$ acts isometrically on a bounded valence, bushy tree $T'$ and there is a quasi-conjugacy $r:T'\rightarrow T$.
\end{thm}

Proposition \ref{prop:gps qi to TxT} can be deduced as a special case of the following:
\begin{thm}\label{thm:qi to product contains alnorm}
Let $G$ be a finitely generated group quasi-isometric to a metric space $X=Z\times \Pi_{i=1}^k X_i$ as in Theorem \ref{thm:dehrahm}. Suppose that $X_1$ is a bounded valence bushy tree. Then there exists a finite index subgroup $G_0\leq G$ that contains an almost normal subgroup $H$ quasi-isometric to $Z\times \Pi_{i=2}^k X_i$, such that the quotient space $G_0/H$ is quasi-isometric to a finite valence bushy tree.
\end{thm}
\begin{proof}
Suppose $f:G\rightarrow X$ is a quasi-isometry. Conjugating the left action of $G$ on itself by $f$ defines a quasi-action of $G$ on $X$. By Lemma \ref{lem:qa on direct factor}, we deduce that a subgroup  $G_0\leq G$ of index at most $k!$ admits a quasi-action on $X_1$. By Theorem \ref{thm:mswthm1}, $G_0$ acts on isometrically on a tree $T$ and this action is quasi-conjugate to the quasi-action of $G_0$ on $X_1$. Let $H$ be a stabiliser of an edge of $T$. Since $T$ is a bounded valence tree, every conjugate of $H$ is commensurable to $H$ and so $H\alnorm G_0$. The quotient space $G_0/H$ is quasi-isometric to $T$. Using the quasi-conjugacy $T\rightarrow X_1$ and Theorem \ref{thm:dehrahm}, we deduce $f(H)$ has finite Hausdorff distance from $p_1^{-1}(x)$ for some $x\in X$, demonstrating that $H$ is quasi-isometric to $Z\times \Pi_{i=2}^k X_i$.
\end{proof}